\documentclass{imsart}

\RequirePackage[OT1]{fontenc}
\RequirePackage{amsthm,amsmath,natbib,amssymb,mathrsfs,amsfonts}
\RequirePackage[colorlinks,citecolor=blue,urlcolor=blue]{hyperref}
\usepackage{txfonts}

\startlocaldefs
\numberwithin{equation}{section}
\theoremstyle{plain}
\newtheorem{thm}{Theorem}[section]
\theoremstyle{remark}

\endlocaldefs
\numberwithin{equation}{section}

\newtheorem{lemma}[thm]{Lemma}
\newtheorem{remark}[thm]{Remark}

\newtheorem{col}[thm]{Corollary}

\newcommand{\bb}{{\bf B}}
\newcommand{\bs}{{\bf S}}
\newcommand{\br}{{\bf r}}
\newcommand{\bi}{{\bf I}}
\newcommand{\bx}{{\bf X}}
\newcommand{\be}{{\bf e}}
\newcommand{\bS}{{\boldsymbol\Sigma}}
\newcommand{\bd}{{\mathbb{R}}}
\newcommand{\ba}{{\bf A}}
\newcommand{\bh}{{\bf H}}
\newcommand{\re}{{\rm E}}
\newcommand{\rtr}{{\rm tr}}

\renewcommand{\(}{\left(}
\renewcommand{\)}{\right)}

\newcommand{\mb}{\mathbf}

\allowdisplaybreaks
\begin{document}

\begin{frontmatter}
\title{No eigenvalues outside the limiting support of the spectral distribution of general sample covariance matrices}
\runtitle{No eigenvalues outside the support}

\begin{aug}
\author{ \snm{Yanqing Yin}\thanksref{a}\ead[label=e1]{yinyq@jsnu.edu.cn}}

\thanks{ Yanqing, Yin was partially supported by NSFC 11701234 and the Priority Academic Program Development of Jiangsu Higher Education Institutions.
}

\address[a]{{School of Mathematics and Statistics, Jiangsu Normal University, Xuzhou, P.R.C., 221116.}
\printead{e1}}



\end{aug}

\begin{abstract}
This paper is to investigate the spectral properties of sample covariance matrices under a more general population. We consider a class of matrices of the form $\mb S_n=\frac1n\bb_n\bx_n\bx_n^*\bb_n^*$, where $\mb B_n$ is a $p\times m$ non-random matrix and $\mb X_n$ is an $m\times n$ matrix consisting of i.i.d standard complex entries. $p/n\to c\in (0,\infty)$ as $n\to \infty$ while $m$ can be arbitrary. We proved that under some mild assumptions, with probability 1, there will be no eigenvalues in any closed interval contained in an open
interval outside the supports of the limiting distribution $F_{c_n,H_n}$, for all sufficiently large $n$. An extension of Bai-Yin law is also obtained.
\end{abstract}


\end{frontmatter}
\section{Introduction}
\subsection{Background and  motivation}
The analysis of the properties of sample covariance matrix, which plays an important role in multivariate analysis as well as high-dimensional data, has been paid impressive attentions. Suppose we observe $\mb y_1,\mb y_2,\cdots,\mb y_n$, independently and identically distribute (i.i.d.) $p$-dimensional complex random variables with mean vector $\mb 0_p$ and covariance matrix $\mb \Sigma_p$ and denote $$\mb y_i=\(y_{1,i},y_{2,i},\cdots,y_{p,i}\)',\quad i=1,2,\cdots,n,$$ where `$'$' stands for the ordinary transpose of a vector. Many statistical problems such as Principal Component Analysis (PCA), the estimates of population covariance matrices and the tests that involve covariance matrices, require the investigation of the spectrum of sample covariance matrices, which is defined as $$\mb S_n=\frac{1}{n}\sum_{i=1}^n \mb y_i\mb y_i^*=\frac{1}{n}\mb Y_n\mb Y_n^*,$$ here `$*$' stands for conjugate transpose and $\mb Y_n=\(\mb y_1,\mb y_2,\cdots,\mb y_n\)$ is the observation matrix. In classical multivariate analysis, the theory of sample covariance matrices for normal variables has been well worked out, see for instance the famous book of  \cite{Anderson1983An}. However, it became apparent that multivariate data in practice were rarely Gaussian. What is more, even in Gaussian case, the exact expressions of results were too cumbersome. The asymptotic analysis when the dimension $p$ is fixed while the sample size $n$ tends to infinity was shown to be convenient and thus been applied extensively for a long time. In this classical framework, the sample covariance matrices $\mb S_n$ can be viewed as a good estimator of population covariance matrix $\mb \Sigma_p$ since the spectrum of $\mb S_n$ are consistent estimators of the spectrum of $\mb \Sigma_p$. In fact, by law of large numbers, for any $1\leq j,k\leq p$, the element lies in the $j$-th row and $k$-th column of $\mb S_n$, denoted by $s_{j,k,n}$ is a consistent estimator of the corresponding element $\sigma_{j,k}$ of $\mb \Sigma_p$. Then, according to the theory of matrix perturbation, for any $j$, the distance between the $j$-th largest eigenvalues of $\mb S_n$ and $\mb \Sigma_p$ is $o(p)$, which tends to 0 as $n\to\infty$ when $p$ is fixed.

However, statisticians are facing with datasets of increasingly larger size nowadays and
the practical relevance of classical framework that $p$ is fixed and $n$ goes to infinity is often unreasonably erroneous. One of the solutions to this challenge is to develop a framework of asymptotic theory  that both $p$ and $n$ tend to infinity. Obviously, for all $j$, the distance between the $j$-th largest eigenvalues of $\mb S_n$ and $\mb \Sigma_p$ may be constant order or even diverge in this case. But what is the exact relationship between the spectrums of $\mb S_n$-the sample covariance matrix and $\mb \Sigma_p$-the population covariance matrix? The pioneer work of  \cite{marchenko1967distribution}, continued in \cite{Wachter1978The,Silverstein1995Strong}, give out some fundamental answers to this question.
Let $\ba$ be any $n \times n$ square matrix having real eigenvalues and denote its eigenvalues  by ${\lambda_j}, j = 1,2, \cdots, n$. The empirical spectral distribution (ESD) of $\ba$ is defined by
$$F^{\ba}\left(x\right) =\frac{1}{n}\sum\limits_{j = 1}^n {I\left({\lambda _j} \le x\right)},$$
where ${I\left(D\right)}$ is the indicator function of an event ${D}$ and the Stieltjes transform of ${F^{\mathbf A}}\left(x\right)$ is given by
$${ m}_{F^{\ba}}\left(z\right)=\int_{-\infty}^{+\infty}\frac{1}{x-z}d{F^{\mathbf A}}\left(x\right),$$
where $z=u+ iv\in\mathbb{C}^+$. The famous M-P law states that if $\mb Y_n=\mb \Sigma_p^{1/2}\mb X_n$, where $\mb \Sigma_p^{1/2}$ is the $p\times p$ Hermitian square root of $\mb \Sigma_p$ and $\mb X_n=\(\mb x_1,\mb x_2,\cdots,\mb x_n\)$ is a $p\times n$ matrix whose elements are i.i.d standard complex random variables, $p/n\to c\in(0,\infty),$ $F^{\bS_p}\xrightarrow{d}H$ and the sequence $\(\mb \Sigma_p\)_p$ is bounded in spectral norm, then almost surely, the ESD $F^{\bs_n}$ of the sample covariance matrix $\mb S_n=\frac{1}{n}\mb \Sigma_p^{1/2}\mb X_n\mb X_n^*\mb \Sigma_p^{1/2}$, tends weakly to a nonrandom p.d.f. $F$ as $n\to \infty$. And for each $z\in\mathbb{C}^+$, $m(z)=m_F(z)$ is a solution to the equation
\begin{align}\label{al3}
m(z)=\int\frac1{t(1-c-czm(z))-z}dH(t),
\end{align}
which is unique in the set $\left\{m(z)\in\mathbb{C}^+: -(1-c)/z+cm(z)\in\mathbb{C}^+\right\}$. When $c\in(0,1)$ and $\mb \Sigma_p^{1/2}=\mb I_p$, the $p$ dimensional identity matrix, it can be derived from M-P law that the support of the limit of $F^{\mb S_n}$ is $[(1-\sqrt{c})^2,(1+\sqrt{c})^2]$. Note that all eigenvalues of $\mb \Sigma_p$ are equal to 1, apparently, the spectrum of $\mb S_n$ is no longer a good estimator of $\mb \Sigma_p$.
A consequent question one may ask is: what if there is no exact structure on $\mb Y_n$ ? That is to say, whether the M-P law valid for more general observations $\mb y_1,\mb y_2,\cdots,\mb y_n$ with mean vector $\mb 0_p$ and population covariance matrix $\mb \Sigma_p$.  \cite{Bai2008Large} gives an sufficient condition which ensure the valid of M-P law and an example where M-P limit is  failure was showed in \cite{Li2017On} recently.

The investigation of ESD is just the first step. By definition, finite outliers will not change the asymptotic behavior of ESD. Thus the next important problem is about the convergence of the extreme eigenvalues of $\mb S_n$. That is to say, whether the extreme eigenvalues of $\mb S_n$ tend to the edge of the limit spectral distribution (LSD). What is more, we shall take one more step further and ask whether there are eigenvalues outside the support of the LSD of $\mb S_n$. In \cite{Yin1988,Bai1988A} the so-called  Bai-Yin law was proved, which state that for $\mb S_{n,0}=\frac{1}{n}\mb X_{n,0}\mb X_{n,0}^*$, where $\mb X_{n,0}$ is a $p\times n$ matrix of the upper-left corner of a standard i.i.d double array $\{x_{j,k}\}$, the necessary and sufficient condition for almost surely (a.s.) convergence of the largest eigenvalue to $(1+\sqrt{c})^2$ is the existence of the fourth moment. The sufficient condition for almost surely convergence of the smallest eigenvalue of $\mb S_{n,0}$ to $(1-\sqrt{c})^2$ when $c\in(0,1)$ was given in  \cite{BaiYin1993} while the necessary condition was proved by  \cite{Tikhomirov2015The} recently. As far as we know, the most recent paper consider the convergence of the extreme eigenvalue of sample covariance matrices comes from \cite{Chafa2017On}. They showed the convergence in probability of the extreme eigenvalues of $\mb S_n$ when the population covariance matrix $\mb \Sigma_p=\mb I_p$ and the distribution law of $\mb y_1$ is log-concave. For other related work, we refer the reader to \cite{Pillai2014Universality,Feldheim2010A,Jonsson2008Some,P2009Universality} and references therein. In many applications, the no-eigenvalue result turns out to
be useful. As it can be, for instance, used to deal with
random quantities involving $\mb S_n$ or its inverse. The study of exact separation of spectrum starting from the classical work of  \cite{Bai1998No,Bai1999Exact} and continued in  \cite{Paul2009No}.

\subsection{The model and main results}
In this paper, we consider the following model. Suppose that the observation matrix $\mb Y_n=\(\mb y_1,\mb y_2,\cdots,\mb y_n\)=\mb B_n\mb X_n$. Then the sample covariance matrix $$\bs_n=\frac{1}{n}\sum_{i=1}^n\mb y_i\mb y_i^*=\frac{1}{n}\mb Y\mb Y^*=\frac1n\bb_n\bx_n\bx_n^*\bb_n^*.$$ We assume:
\begin{itemize}
\item[(a)] $\bx_n=(x_{jk})$ is an $m\times n$ matrix whose entries are i.i.d. complex variables with mean zero and variance $1$, and $m$ can be arbitrary (possibly infinite);
\item[(b)] $\bb_n$ is a $p\times m$ matrix such that $\bb_n\bb_n^*=\bS_p,$ whose spectral norm $\|\mb \Sigma_p\|$ is uniformly bounded;
\item[(c)] For each $p$, $H_n=F^{\bS_p}\xrightarrow{d}H$, a probability distribution function (p.d.f.);
\item[(d)] $c_n=p/n\to c\in(0,\infty)$ as $n\to\infty$;
\item[(e)] For some $0<\delta<1$, $\re|x_{11}|^{6+\delta}=\mu<\infty$.
\end{itemize}

\begin{remark}
In our model, $\bx_n=(x_{jk})$ may be dependent on $n$, i.e., the data  may not comes from a double array.
\end{remark}

\begin{remark}
Under assumption $(a)$ and $(b)$, $\mb Y_n$ can be viewed as an observation matrix that consists of $n$ samples drawn from a $p$-dimensional population with mean vector $\mb 0$ and covariance matrix $\mb \Sigma_n$. One may find that when $m=p$, our model reduce to the well studied model in the sense of spectrum due to the polar decomposition.
\end{remark}

\begin{remark}
This model covers variety population from time series. Such as the repeated linear process arises in panel surveys or longitudinal studies where
$$\mb y=\mb B\mb x=\left(
                     \begin{array}{cccccccc}
                       \cdots & b_{1-p} & b_{2-p} & \cdots & b_{0} & b_{1} & b_{2}& \cdots \\
                       \cdots & b_{2-p} & b_{3-p} & \cdots & b_{1} & b_{2} & b_{3} & \cdots \\
                       \vdots & \vdots & \vdots & \vdots & \vdots & \vdots & \vdots & \vdots\\
                       \cdots & b_{-1} & b_{0} & \cdots & b_{p-2}& b_{p-1} & b_{p} & \cdots \\
                       \cdots & b_{0} & b_{1} & \cdots & b_{-1} & b_{p} & b_{p+1} & \cdots\\
                     \end{array}
                   \right)
                   \left(
                     \begin{array}{c}
                       \vdots \\
                       x_p \\
                       \vdots \\
                       x_1 \\
                       \vdots \\
                     \end{array}
                   \right).
$$
\end{remark}

Our first result is as following
\begin{thm}\label{tmf}
Under assumption (a-e), then as $n\to \infty$, the ESD $F^{\bs_n}$ of the sample covariance matrix $\mb S_n$ tends weakly to a nonrandom p.d.f. $F$  whose Stieltjes transform $m(z)=m_F(z)$ satisfies equation (\ref{al3}).
\end{thm}

Let $\underline\bs_n=\frac1n\bx_n^*\bb_n^*\bb_n\bx_n$, then we know that the nonzero eigenvalues of $\bs_n$ and $\underline\bs_n$ are the same. It is easy to verify
\begin{align*}
F^{\underline\bs_n}(x)=(1-c_n)I_{[0,\infty)}+c_nF^{\bs_n}(x)
\end{align*}
which implies that
\begin{align}\label{al2}
m_{F^{\underline\bs_n}}(z)=-(1-c_n)/z+c_nm_{F^{\bs_n}}(z),\quad z\in\mathbb{C}^+.
\end{align}
Therefore we know that the limiting spectral distribution (LSD) $\underline F$ of $F^{\underline\bs_n}$ exists and satisfies
\begin{align*}
\underline F(x)=(1-c)I_{[0,\infty)}+cF(x).
\end{align*}
Thus
\begin{align*}
m_{\underline F}(z)=-(1-c)/z+cm_F(z).
\end{align*}
Due to (\ref{al3}) and the above equality, we find
\begin{align*}
m(z)=-z^{-1}\int\frac1{t\underline m(z)+1}dH(t)
\end{align*}
and
\begin{align}\label{al6}
{\underline m(z)}=-\left({z-\int\frac t{t\underline m(z)+1}dH(t)}\right)^{-1}
\end{align}
where $\underline m(z)=m_{\underline F}(z)$. If we let $F^{c,H}$ denote $\underline F$, then $F^{c_n,H_n}$ can be viewed as the limiting nonrandom p.d.f. associated with ratio $c_n$ and $H_n$.

The most important theorem of this paper states

\begin{thm}\label{tms}
Under the above model and assuming (a-e) with one more assumptions
\begin{itemize}
\item[(f)] The interval $[a,b]$ with $a>0$ lies outside the support of $F^{c_n,H_n}$ for all large n,
\end{itemize}
we have
\begin{align*}
{\rm P}\left({\rm no \ eigenvalues \ of} \ \bs_n \ {\rm appears \ in \ [a,b] \ {\rm for \ all \ large \ } n}\right)=1.
\end{align*}
\end{thm}

\begin{remark}
From the next section, one may find that also the strategy of the proof of Theorem \ref{tms} is similar with \cite{Bai1998No}, the results in the present paper are definitely non-trivial extension of the existing results. In fact, the procedure used in \cite{Bai1998No} strongly depends on the foundational results proved in \cite{Yin1988}, which can not be applied in our model  directly. Thus we can not truncate the variables at a finite number. What is more, we need to find a sufficient low bound for the largest eigenvalue of $\mb S_n$, which has not been studied before. This is achieved by combining the Non-asymptotic theory of random matrix, see for instance \cite{ver2010}, which is an elementary interplay between probability and linear algebra, and an inequality for a kind of quadratic forms proved in this paper.
\end{remark}

As a direct application of above theorems, we have
\begin{thm}\label{tmt}
  Under assumptions (a-e) and  $\mb \Sigma_p=\mb I_p$, the largest eigenvalue of $\mb S_n$ convergences to $(1+\sqrt{c})^2$ almost surely.
\end{thm}

\begin{remark}
This theorem can be viewed as an extension of Bai-Yin law in a different direction with \cite{Chafa2017On}.
\end{remark}

The rest of the paper is organised as follows. The next section is to give the proof of the main theorems while the last section lists some necessary Lemmas. We also note here that throughout this paper, $C$ stands for a constant that may take different values from one appearance to another.


\section{Proofs of the main theorems}

\subsection{Proof of theorem \ref{tmf} }
We firstly prove the theorem of LSD.  Noticing the results of Theorem 1.1. in \cite{Bai2008Large}, it is sufficient to verify the moment condition, i.e, for all $j$ and any non-random $p\times p$ matrix $\mb A$ with bounded norm,
$$\re|\mb x_j^*\bb_n^*\ba\bb_n\mb x_j-\rtr\ba\bS_p|^2=o(n^2).$$
This is in fact proved by the arguments before (\ref{al10}) in the proof of Lemma \ref{le3}.
\subsection{Proof of theorem \ref{tms} }
The main tools used in the proof are bounds on the moments of martingale difference sequences and properties of Stieltjes transform as well as some results in non-asymptotic analysis of random matrices.

The first step of the proof aims at truncating the variable of $\mb X_n$ as we need to deal with $x_{ij}^p$ for some $p$ larger than $6$. In \cite{Bai1998No}, under the model assumption $m=p$, the authors truncated the variables in $\mb X_{n,0}$ at $C$, a sufficiently large but finite number. This is due to Bai-Yin law states that the limit of the spectral norm of $\frac{1}{n}\mb X_{n,0}\mb X_{n,0}$ is $(1+\sqrt{c})^2$. However, this is invalid under the model of the present paper since $m$ can be arbitrary.
\subsubsection{Truncation, Centralization and Rescaling}\label{truncation}
Let $\bb_n=(b_{jk})$ and $b_{k}=\sqrt{\sum_{j=1}^p|b_{jk}|^2}$, it follows that
\begin{align*}
b_j\leq C_1, \ \quad \sum_{k=1}^mb_k^2\le p\|\bS_p\|\le C_2p.
\end{align*}
At first, we shall truncate the variables $x_{jk}$ at $n^{1/2-\eta}/b_{ j},j=1,\cdots,m,k=1,\cdots,n$ for $\eta=\frac\delta6/(6+\delta)\in(0,1/42)$.

Define $\hat x_{jk}=x_{jk}I(|x_{jk}|\le n^{1/2-\eta}/b_{j} ),\widehat\bx_n=(\hat x_{jk})$ and $\widehat\bs_n=\frac1n\bb_n\widehat\bx_n\widehat\bx_n^*\bb_n^*$. Applying assumption $(e)$, we get
\begin{align*}
P(\bx_n\neq \widehat\bx_n,i.o.)\le&\lim_{N\to\infty}\sum_{n=N}^{\infty}P(\bx_n\neq \widehat\bx_n)\le\lim_{N\to\infty}\sum_{n=N}^{\infty}\sum_{j=1}^m\sum_{k=1}^nP(x_{jk}\neq \widehat x_{jk})\\
\le&\lim_{N\to\infty}\sum_{n=N}^{\infty}\sum_{j=1}^m\sum_{k=1}^nP(|x_{jk}|>n^{1/2-\eta}/b_{ j} )\\
\le&\lim_{N\to\infty}\sum_{n=N}^{\infty}\sum_{j=1}^m\sum_{k=1}^n n^{-(1/2-\eta)(6+\delta)}b_{ j}^{6+\delta}\re|x_{jk}|^{6+\delta}\\
\le&\lim_{N\to\infty}\sum_{n=N}^{\infty}\sum_{j=1}^m\sum_{k=1}^n n^{-3-\delta/3}b_{ j}^{6+\delta}\re|x_{jk}|^{6+\delta}\\
\le&C_1^{4+\delta}\mu\lim_{N\to\infty}\sum_{n=N}^{\infty} n^{-2-\delta/3}\sum_{j=1}^m b_{ j}^{2}
\le C\lim_{N\to\infty}\sum_{n=N}^{\infty} n^{-1-\delta/3}\to0.
\end{align*}

Denote the eigenvalues of $\widehat\bs_n$ and $\widehat\bs_n-\re\widehat\bs_n$ by $\lambda_k$ and $\hat{\lambda_k}$ (in decreasing order), then $\lambda_k^{1/2}$ and $\hat{\lambda_k}^{1/2}$ are the $k$th largest singular values of $\frac1{\sqrt n}\bb_n\widehat\bx_n$ and $\frac1{\sqrt n}\bb_n\left(\widehat\bx_n-\re\widehat\bx_n\right)$ respectively. By Lemma \ref{le1}, one finds
\begin{align*}
\max_{k\le n}|\lambda_k^{1/2}-\hat\lambda_k^{1/2}|\le&\frac1{\sqrt n}\left\|\bb_n\widehat\bx_n-\bb_n\left(\widehat\bx_n-\re\widehat\bx_n\right)\right\|=\frac1{\sqrt n}\left\|\bb_n\re\widehat\bx_n\right\|\\
=&\sqrt{\frac1{ n}\left\|\bb_n\re\widehat\bx_n\re\widehat\bx_n^*\bb_n^*\right\|}\le\sqrt{\frac1{ n}\rtr\left(\bb_n\re{\widehat\bx}_n\re{\widehat\bx}_n^*\bb_n^*\right)}\\
=&\sqrt{\frac1n\sum_{l=1}^p\sum_{k=1}^n\sum_{j_1,j_2=1}^mb_{l{j_1}}\overline b_{lj_2}\re\hat x_{j_1k}\re\overline{\hat x}_{j_2k}}\\
=&\sqrt{\sum_{l=1}^p\sum_{j_1,j_2=1}^mb_{l{j_1}}\overline b_{lj_2}\re\hat x_{j_11}\re\overline{\hat x}_{j_21}}\\
\le&\left\{\sum_{j_1,j_2=1}^m\left|\re\hat x_{j_11}\right|\left|\re{\hat x}_{j_21}\right|\sqrt{\sum_{l=1}^p |b_{l{j_1}}|^2}\sqrt{\sum_{l=1}^p| b_{lj_2}|^2}\right\}^{1/2}\\
=&\left\{\sum_{j_1,j_2=1}^m\left|\re\hat x_{j_11}\right|\left|\re{\hat x}_{j_21}\right|b_{j_1}b_{j_2}\right\}^{1/2}.
\end{align*}
Note that
\begin{align*}
\left|\re\hat x_{j_11}\right|\le&\re\left|x_{j_11}\right|I\left(|x_{j_11}|>n^{1/2-\eta}/b_{j_1}\right)
\le\frac{b_{j_1}^{5+\delta}}{n^{5/2+\delta/3}}\re\left|x_{j_11}\right|^{6+\delta}\le C_1^{4+\delta}\mu\frac{b_{j_1}}{n^{5/2}}.
\end{align*}
Thus, we have
\begin{align*}
\max_{k\le n}|\lambda_k^{1/2}-\hat\lambda_k^{1/2}|\le&\frac{C}{n^{5/2}}\left\{\sum_{j_1,j_2=1}^mb_{j_1}^2b_{j_2}^2\right\}^{1/2}\le Cn^{-3/2}.
\end{align*}

Write $\widetilde\bx_n=(\tilde x_{jk})$ and $\widetilde\bs_n=\frac1n\bb_n\widetilde\bx_n\widetilde\bx_n^*\bb_n^*$ where $\tilde x_{jk}=\left(\hat x_{jk}-\re\hat x_{jk}\right)/\sigma_j$ and $\sigma_j^2=\re\left|\hat x_{j1}-\re\hat x_{j1}\right|^2$. Using assumption $(e)$, one gets
\begin{align}\label{al1}
1-\sigma_j^2=&\re\left|x_{j1}\right|^2I\left(|x_{j1}|\ge n^{1/2-\eta}/b_j\right)+\left|\re x_{j1}I\left(|x_{j1}|\ge n^{1/2-\eta}/b_j\right)\right|^2\\
\le&2\re\left|x_{j1}\right|^2I\left(|x_{j1}|\ge n^{1/2-\eta}/b_j\right)\notag\\
\le&\frac{2b_j^{4+\delta}}{n^{2+\delta/3}}\re\left|x_{j1}\right|^{6+\delta}\le\frac{2\mu b_j^{4+\delta}}{n^{2}}\le\frac{Cb_j^{2}}{n^{2}}\notag.
\end{align}
Let $\tilde\lambda_k$ denote the $k$th largest eigenvalues of $\widetilde\bs_n$, then it follows that by Lemma \ref{le1} and (\ref{al1}),
\begin{align*}
&\max_{k\le n}|\hat\lambda_k^{1/2}-\tilde\lambda_k^{1/2}|\le\frac1{\sqrt n}\left\|\bb_n\left(\widehat\bx_n-\re\widehat\bx_n-\widetilde\bx_n\right)\right\|\\
\le&\sqrt{\frac1{n}\rtr\left[\bb_n\left(\widehat\bx_n-\re\widehat\bx_n-\widetilde\bx_n\right)\left(\widehat\bx_n-\re\widehat\bx_n-\widetilde\bx_n\right)^*\bb_n^*\right]}\\
=&\sqrt{\frac1{n}\sum_{l=1}^p\sum_{k=1}^n\sum_{j_1,j_2=1}^mb_{lj_1}\overline b_{lj_2}(1-\sigma_{j_1}^{-1})(1-\sigma_{j_2}^{-1})
\left(\hat x_{j_1k}-\re\hat x_{j_1k}\right)\left(\hat x_{j_2k}-\re\hat x_{j_2k}\right)}\\
\le&\sqrt{\sum_{l=1}^p\sum_{j_1,j_2=1}^m\frac{4n^{1-2\eta}}{b_{j_1}b_{j_2}}|b_{lj_1}b_{lj_2}||(1-\sigma_{j_1}^{-1})(1-\sigma_{j_2}^{-1})}|\\
\le&\left\{\sum_{j_1,j_2=1}^m\frac{4n^{1-2\eta}}{b_{j_1}b_{j_2}}|(1-\sigma_{j_1}^{-1})(1-\sigma_{j_2}^{-1})|\sqrt{\sum_{l=1}^p |b_{lj_1}|^2}
\sqrt{\sum_{l=1}^p |b_{lj_2}|^2}\right\}^{1/2}\\
\le&\left\{\sum_{j_1,j_2=1}^m{n}(1-\sigma_{j_1}^{2})(1-\sigma_{j_2}^{2})\right\}^{1/2}
\le\frac C{n^{3/2}}\left\{\sum_{j_1,j_2=1}^m{b_{j_1}^{2}}{b_{j_2}^{2}}\right\}^{1/2}\le\frac C{n^{1/2}}\to0.
\end{align*}

For simplicity, the truncated and recentralized variables are still denoted by $x_{jk}$. We assume in the following
\begin{itemize}
\item[(1)] The variables $\{x_{jk},j=1,2,\cdots,m;k=1,2,\cdots,n\}$ are independent.
\item[(2)] $\re(x_{jk})=0$ and ${\rm Var}(x_{jk})=1$.
\item[(3)] $|x_{jk}|\le n^{1/2-\eta}/b_j$.
\item[(4)] $\sup_{n,j,k}\re|x_{jk}|^{6+\delta}\le M$.
\end{itemize}

\subsubsection{A primary bound on the largest eigenvalue of $\bs_n$}

This part is to give a primary bound on the largest eigenvalue of $\bs_n$. We need the following Lemma
\begin{lemma}[Theorem 5.44 in \cite{ver2010}]\label{vsn}
let $\bf A$ be an $N\times M$ matrix whose row $A_i$ are independent random row vectors in $\mathbb{C}^M$ with the common second moment matrix $\bf \Sigma=\rm E A_i'A_i$. Let $l$ be a number such that $\sqrt{A_iA_i'}\leq \sqrt{l}$ almost surely for all $i$. Then for every $t>0$, the following inequality holds with probability at least $1-n\exp^{-Ct^2}$:
\begin{align}
  \|{\bf A}\|\leq {\bf\|\Sigma\|^{1/2}}\sqrt{N}+ t \sqrt{l},
\end{align}
here $C$ is a constant.
\end{lemma}

\begin{remark}
The original theorem in fact consider the real case, however, as indicated in the corresponding paper, one can easily adjust it to the complex case.
\end{remark}

Let $\bf A=\frac{1}{\sqrt{n}}\bf X_n^*\bf B_n^*$, $N=n$, $M=p$. We have $\|\bf \Sigma\|^{1/2}=\|\rm E A_iA_i'\|^{1/2}\leq C$. Then apply Lemma \ref{le3}, we have for any $i$
\begin{align*}
{\rm P}({|\mb x_i^*{\bf B_n^*}{\bf B_n}\mb x_i-\rtr{\bf B_n}{\bf B_n^*}|>n})\leq \frac{{\rm E}|\mb x_i^*{\bf B_n^*}{\bf B_n}\mb x_i-\rtr{\bf B_n}{\bf B_n^*}|^l}{n^l}\leq \frac{C_l}{n^{\eta l}}
\end{align*}
Choosing $l=\frac{2}{\eta}+1$, we have $${\rm P}(\sup_{i}{A_iA_i'>Cn})\leq \frac{n}{n^{2+\eta}},$$ which is summable.

Then letting $t=\sqrt{\log n}$ and $l=n$, by Lemma \ref{vsn}, we arrive at for any $s$,
\begin{align}\label{al11}
{\rm P}(\|{\bf S_n}\|>K\log n)=o(n^{-s}),
\end{align}
here $K$ is a constant only depend on $s$.
\subsubsection{Convergence of random part}
Let $m_n(z)=m_{F^{\bs_n}}(z)$, $\underline m_n(z)=m_{F^{\underline\bs_n}}(z)$ and $\underline m_n^0(z)=m_{F^{c_n,H_n}}(z)$, the goal of this part is to show that for $z=u+iv_n=u+in^{-3\eta/98}$
\begin{align}\label{al18}
\sup_{u\in[a,b]}nv_n|m_n(z)-\re m_n(z)|\xrightarrow{a.s.}0,\quad n\to\infty.
\end{align}
To begin with, we introduce some notations. For $j=1,2,\cdots,n$, denote
\begin{align*}
&{\bf x}_j=(x_{1j},x_{2j},\cdots,x_{mj})',\quad \br_j=\frac1{\sqrt n}\bb_n{\bf x}_j,\quad \bs_{nj}=\bs_n-\br_j\br_j^*,\\
&\bd_n=\bs_n-z\bi_p,\quad \bd_{nj}=\bd_n-\br_j\br_j^*,\quad \bd_{nj\underline j}=\bd_{nj}-\br_{\underline j}\br_{\underline j}^*,\\
&\phi_j=\br_j^*\bd_{nj}^{-1}\br_j-\frac1n\rtr\left(\bd_{nj}^{-1}\bS_p\right),\quad \hat\phi_j=\br_j^*\bd_{nj}^{-1}\br_j-\frac1n\re\rtr\left(\bd_{nj}^{-1}\bS_p\right),\\
&\xi_j=\br_j^*\bd_{nj}^{-2}\br_j-\frac1n\rtr\left(\bd_{nj}^{-2}\bS_p\right),\quad \hat\phi_{j\underline j}=\br_{\underline j}^*\bd_{nj\underline j}^{-1}\br_{\underline j}-\frac1n\re\rtr\left(\bd_{nj\underline j}^{-1}\bS_p\right).
\end{align*}
Also let
\begin{align*}
&\rho_j=\frac1{1+\br_j^*\bd_{nj}^{-1}\br_j},\quad \hat\rho_j=\frac1{1+n^{-1}\rtr(\bd_{nj}^{-1}\bS_p)},\quad b_n=\frac1{1+n^{-1}\re\rtr(\bd_{n1}^{-1}\bS_p)},\\
&\rho_{j\underline j}=\frac1{1+\br_{\underline j}^*\bd_{nj\underline j}^{-1}\br_{\underline j}},\quad b_{1n}=\frac1{1+n^{-1}\re\rtr(\bd_{n,1,2}^{-1}\bS_p)}.
\end{align*}

Firstly, we want to find the bounds of $\re\rho_1$ and $b_n$. For this purpose, we need the bounds on moments of $\phi_j$ and $\hat\phi_j$ for $j=1,\cdots,n$.

Using Lemma \ref{le3}, we have for $l\ge1$,
\begin{align*}
\re|\phi_j|^{2l}=\frac1{n^{2l}}\re\left|{\bf x}_j^*\bb_n^*\bd_{nj}^{-1}\bb_n{\bf x}_j-\rtr\left(\bd_{nj}^{-1}\bS_p\right)\right|^{2l}\le C_ln^{-2\eta l}v_n^{-2l}.
\end{align*}
Let $\re_0(\cdot)$ denote expectation and $\re_k(\cdot)$ denote conditional expectation with respect to the $\sigma$-field generated by $\br_1,\cdots,\br_k$. Applying Lemma \ref{le4} and the identity
\begin{align}\label{al4}
\left(\ba+\br\br^*\right)^{-1}=\ba^{-1}-\frac{\ba^{-1}\br\br^*\ba^{-1}}{1+\br^*\ba\br},
\end{align}
it follows that for $l\ge1$,
\begin{align}\label{al26}
&\re\left|\hat\phi_j-\phi_j\right|^{2l}=\re\left|\hat\phi_1-\phi_1\right|^{2l}
=n^{-2l}\re\left|\sum_{j=2}^n\left[\re_j\rtr\left(\bS_p\bd_{n1}^{-1}\right)-\re_{j-1}\rtr\left(\bS_p\bd_{n1}^{-1}\right)\right]\right|^{2l}\\
=&n^{-2l}\re\left|\sum_{j=2}^n\left(\re_j-\re_{j-1}\right)\rtr\left[\bS_p\left(\bd_{n1}^{-1}-\bd_{n1j}^{-1}\right)\right]\right|^{2l}\notag
=n^{-2l}\re\left|\sum_{j=2}^n\left(\re_j-\re_{j-1}\right)\frac{\br_j^*\bd_{n1j}^{-1}\bS_p\bd_{n1j}^{-1}\br_j}{1+\br_j^*\bd_{n1j}^{-1}\br_j}\right|^{2l}\notag\\
\le&C_ln^{-2l}\re\left(\sum_{j=2}^n\left|\left(\re_j-\re_{j-1}\right)\frac{\br_j^*\bd_{n1j}^{-1}\bS_p\bd_{n1j}^{-1}\br_j}{1+\br_j^*\bd_{n1j}^{-1}\br_j}\right|^2\right)^{l}\notag\\
\le&C_ln^{-2l}\re\left[\sum_{j=2}^n\left(\frac{\left|\br_j^*\left(\bd_{n1j}^*\right)^{-1}\bd_{n1j}^{-1}\br_j\right|}{\Im\left(\br_j^*\bd_{n1j}^{-1}\br_j\right)}\right)^2\right]^{l}
\le C_ln^{-l}v_n^{-2l}.\notag
\end{align}
From this we know, for $l\ge1$,
\begin{align}\label{al23}
\re\left|\hat\phi_j\right|^{2l}\le C_ln^{-2\eta l}v_n^{-2l}.
\end{align}

Using (\ref{al5}) and (\ref{al13}) in \ref{se2}, we obtain that for large $n$
\begin{align*}
\sup_{u\in[a,b]}|\re\rho_1|=\sup_{u\in[a,b]}|z\re\underline m_n(z)|\le&\sup_{u\in[a,b]}|z\re\underline m_n(z)-z\underline m_n^0(z)|+\sup_{u\in[a,b]}|z\underline m_n^0(z)|\\
=&o(v_n)+C\le C+1.
\end{align*}
It is known that $\rho_j$, $\hat\rho_j$, and $b_n$ all bounded in absolute value by $|z|/v_n$. Noticing $b_n=\rho_1+\rho_1b_n\hat\phi_1$, one finds
\begin{align*}
\sup_{u\in[a,b]}|b_n|=&\sup_{u\in[a,b]}\left|\re\rho_1+\re\rho_1b_n\hat\phi_1\right|\\
\le&\sup_{u\in[a,b]}\left|\re\rho_1\right|+\sup_{u\in[a,b]}\frac{|z|^2}{v_n^2}\re^{1/2}\left|\hat\phi_1\right|^2\\
\le&(C+1)+\sup_{u\in[a,b]}\frac{|z|^2}{n^{1/2}v_n^3}\le C+2\triangleq C_0.
\end{align*}

Next, let $S_n'$ be a set contains $n^2$ elements, equally spaced in $[a,b]$. Note that $\left|m_n(u_1+iv_n)-m_n(u_2+iv_n)\right|\le|u_1-u_2|v_n^{-2}$, our goal follows if we can show
\begin{align*}
\max_{u\in S_n'}nv_n\left|m_n(z)-\re m_n(z)\right|\xrightarrow{a.s.}0.
\end{align*}

Write
\begin{align*}
m_n(z)-\re m_n(z)=&\frac1p\sum_{k=1}^n\left(\re_{k-1}\rtr\bd_n^{-1}-\re_k\rtr\bd_n^{-1}\right)=-\frac1p\sum_{k=1}^n\left(\re_{k-1}-\re_k\right)\rho_k\br_k^*\bd_{nk}^{-2}\br_k\\
=&-\frac1p\sum_{k=1}^n\left(\re_{k-1}-\re_k\right)\rho_k\xi_k-\frac1{pn}\sum_{k=1}^n\left(\re_{k-1}-\re_k\right)\rho_k\rtr\bd_{nk}^{-2}\bS_p\\
=&-\frac1p\sum_{k=1}^n\left(\re_{k-1}-\re_k\right)\rho_k\xi_k-\frac1{pn}\sum_{k=1}^n\left(\re_{k-1}-\re_k\right)\hat\rho_k\rtr\bd_{nk}^{-2}\bS_p\\
&+\frac1{pn}\sum_{k=1}^n\left(\re_{k-1}-\re_k\right)\rho_k\hat\rho_k\phi_k\rtr\bd_{nk}^{-2}\bS_p\\
=&-\frac1p\sum_{k=1}^n\left(\re_{k-1}-\re_k\right)\rho_k\xi_k
+\frac1{pn}\sum_{k=1}^n\left(\re_{k-1}-\re_k\right)\rho_k\hat\rho_k\phi_k\rtr\bd_{nk}^{-2}\bS_p\\
\triangleq&W_1+W_2.
\end{align*}
From Assumption $(f)$, an $\underline\varepsilon>0$ exists for which $[a-2\underline\varepsilon,b+2\underline\varepsilon]$ also satisfies $(f)$. Let $a'=a-\underline\varepsilon,b'=b+\underline\varepsilon$, and $F_{nk}$ be the ESD of the matrix $\bs_{nk}$. From $(\ref{al14})$ in \ref{se2} and Lemma \ref{le9}, for any $l>0$, we have almost surely
\begin{align}\label{al15}
\max_{k\le n}\re_k\left(F_{nk}{[a',b']}\right)^l=o_{a.s.}(v_n^{4l}).
\end{align}
Now define
\begin{align*}
G_k=&I\left\{\left[\re_{k-1}\left((F_{nk}{[a',b']}\right)\le v_n^4\right]\bigcap\left[\re_{k-1}\left((F_{nk}{[a',b']}\right)^2\le v_n^8\right]\right\}\\
=&I\left\{\left[\re_{k}\left((F_{nk}{[a',b']}\right)\le v_n^4\right]\bigcap\left[\re_{k}\left((F_{nk}{[a',b']}\right)^2\le v_n^8\right]\right\}.
\end{align*}
It follows that from (\ref{al15})
\begin{align*}
{\rm P}\left(\bigcup_{k=1}^n\{G_k=0\}, \ {\rm i.o.}\right)=0.
\end{align*}
Therefore, we have, for any $\varepsilon>0$,
\begin{align}\label{al16}
&{\rm P}\left(\max_{u\in S_n'}|nv_nW_1|>\varepsilon, \ {\rm i.o.}\right)\\
=&{\rm P}\left(\left\{\max_{u\in S_n'}|nv_nW_1|>\varepsilon\right\}\bigcap\left[\left(\bigcup_{k=1}^n\{G_k=0\}\right)\bigcup\left(\bigcap_{k=1}^n\{G_k=1\}\right)\right], \ {\rm i.o.}\right)\notag\\
\le&{\rm P}\left(\left\{\max_{u\in S_n'}|nv_nW_1|>\varepsilon\right\}\bigcap\left(\bigcap_{k=1}^n\{G_k=1\}\right), \ {\rm i.o.}\right)\notag\\
\le&{\rm P}\left(\max_{u\in S_n'}\left|v_n\sum_{k=1}^n\left(\re_{k-1}-\re_k\right)\rho_k\xi_k G_k\right|>c_n\varepsilon, \ {\rm i.o.}\right)\notag
\end{align}
and
\begin{align*}
&{\rm P}\left(\max_{u\in S_n'}|nv_nW_2|>\varepsilon, \ {\rm i.o.}\right)\\
\le&{\rm P}\left(\left\{\max_{u\in S_n'}|nv_nW_2|>\varepsilon\right\}\bigcap\left(\bigcap_{k=1}^n\{G_k=1\}\right), \ {\rm i.o.}\right)\\
\le&{\rm P}\left(\max_{u\in S_n'}\left|v_n\frac1{p}\sum_{k=1}^n\left(\re_{k-1}-\re_k\right)\rho_k\hat\rho_k\phi_k\rtr\left(\bd_{nk}^{-2}\bS_p\right) G_k\right|>\varepsilon, \ {\rm i.o.}\right).
\end{align*}
Note that for ${u\in S_n'}$
\begin{align*}
&\left|\rho_k\xi_k G_k\right|^2=\left|\rho_k\xi_k G_k\right|^2I(|\rho_k<2C_0|)+\left|\rho_k\xi_k G_k\right|^2I(|\rho_k|\ge2C_0)\\
\le&4C_0^2\left|\xi_k G_k\right|^2+\left(\left|\rho_k\br_k^*\bd_{nk}^{-2}\br_k\right|^2+\left|\frac1n\rho_k\rtr\left(\bd_{nk}^{-2}\bS_p\right)\right|^2\right)I(|\rho_k|^{-1}\le\frac1{2C_0})\\
\le&4C_0^2\left|\xi_k G_k\right|^2+C\left(v_n^{-2}+c_n^2\frac{|z|^2}{v_n^2}v_n^{-4}\right)I(|b_n^{-1}+\hat\phi_k|\le\frac1{2C_0})\\
\le&4C_0^2\left|\xi_k G_k\right|^2+C v_n^{-6}I(|\hat\phi_k|\ge\frac1{2C_0}).
\end{align*}
By Lemma \ref{le3}, Lemma \ref{le6}, and the fact $\{\left(\re_{k-1}-\re_k\right)\rho_k\xi_k G_k\}$ forms a martingale difference sequence, we have for each $u\in S_n'$, $l\ge1$, and $t>(9\eta+49)l/(95\eta)$,
\begin{align*}
&\re\left|v_n\sum_{k=1}^n\left(\re_{k-1}-\re_k\right)\rho_k\xi_k G_k\right|^{2l}\\
\le&C_l\left[\re\left(\sum_{k=1}^n\re_{k-1}\left|v_n\rho_k\xi_k G_k\right|^2\right)^l+\sum_{k=1}^n\re\left|v_n\rho_k\xi_k G_k\right|^{2l}\right]\\
\le&C_l\left[\re\left(\sum_{k=1}^n\left(v_n^2\re_{k-1}\left|\xi_k G_k\right|^2+v_n^{-4}{\rm P}(|\hat\phi_k|\ge\frac1{2C_0})\right)\right)^l+\sum_{k=1}^n|z|^{2l}\re\left|\xi_k \right|^{2l}\right]\\
\le&C_l\left[\re\left(\sum_{k=1}^n\left(v_n^2\re_{k-1}G_k\frac{\rtr\left(\bd_{nk}^{-2}\bS_p(\bd_{nk}^*)^{-2}\bS_p\right)}{n^{2}}+v_n^{-4}{\rm P}(|\hat\phi_k|\ge\frac1{2C_0})\right)\right)^l+\sum_{k=1}^n|z|^{2l}\frac{n^{2(1-\eta)l}\|\bd_{nk}^{-2}\|^{2l}}{n^{2l}}\right]\\
\le&C_l\left[\re\left(\sum_{k=1}^n\left(v_n^2\re_{k-1}G_k\frac{\rtr\left(\bd_{nk}^{-2}(\bd_{nk}^*)^{-2}\right)}{n^{2}}+v_n^{-4}{\rm P}(|\hat\phi_k|\ge\frac1{2C_0})\right)\right)^l+\frac1{n^{\eta l-1}v_n^{4l}}\right]\\
\le&C_l\left[\re\left(\sum_{k=1}^nv_n^2\re_{k-1}G_k\frac{\rtr\left(\bd_{nk}^{-2}(\bd_{nk}^*)^{-2}\right)}{n^{2}}\right)^l+n^{l-1}\sum_{k=1}^nv_n^{-4l}{\rm P}(|\hat\phi_k|\ge\frac1{2C_0})+\frac1{n^{\eta l-1}v_n^{4l}}\right]\\
\le&C_l\left[\re\left(\sum_{k=1}^nv_n^2\re_{k-1}G_k\frac{\rtr\left(\bd_{nk}^{-2}(\bd_{nk}^*)^{-2}\right)}{n^{2}}\right)^l+n^{l-1}\sum_{k=1}^nv_n^{-4l}\re|\hat\phi_k|^{2t}+\frac1{n^{\eta l-1}v_n^{4l}}\right]\\
\le&C_l\left[\re\left(\sum_{k=1}^nv_n^2\re_{k-1}G_k\frac{\rtr\left(\bd_{nk}^{-2}(\bd_{nk}^*)^{-2}\right)}{n^{2}}\right)^l+n^{l}v_n^{-4l}n^{-2\eta t}v_n^{-2t}+\frac1{n^{\eta l-1}v_n^{4l}}\right]\\
\le&C_l\left[\re\left(\sum_{k=1}^nv_n^2\re_{k-1}G_k\frac{\rtr\left(\bd_{nk}^{-2}(\bd_{nk}^*)^{-2}\right)}{n^{2}}\right)^l
+\frac{v_n^{2l}}{n^{95\eta t/49-9\eta l/49-l}}+\frac{v_n^{2l}}{n^{40\eta l/49}}\right]\\
\le&C_l\left[\re\left(\sum_{k=1}^nv_n^2\re_{k-1}G_k\frac{\rtr\left(\bd_{nk}^{-2}(\bd_{nk}^*)^{-2}\right)}{n^{2}}\right)^l +v_n^{2l}
\right].
\end{align*}
Let $\lambda_{kj}$ denote the $j$-th largest eigenvalue of $\bs_{nk}$. By (\ref{al15}), we get
\begin{align*}
&\sum_{k=1}^n\re_{k-1}G_k{\rtr\left(\bd_{nk}^{-2}(\bd_{nk}^*)^{-2}\right)}
=\sum_{k=1}^n\re_{k-1}G_k\sum_{j=1}^p\frac1{\left((\lambda_{kj}-u)^2+v_n^2\right)^2}\\
=&\sum_{k=1}^n\re_{k-1}G_k\left[\sum_{\lambda_{kj}\notin[a',b']}\frac1{\left((\lambda_{kj}-u)^2+v_n^2\right)^2}
+\sum_{\lambda_{kj}\in[a',b']}\frac1{\left((\lambda_{kj}-u)^2+v_n^2\right)^2}\right]\\
\le&\sum_{k=1}^n\re_{k-1}G_k\left[p\underline\varepsilon^{-4}
+pv_n^{-4}F_{nk}([a',b'])\right]\le Cn^2
\end{align*}
Combining the above two inequalities, it yields  for each $u\in S_n'$, $l\ge1$,
\begin{align}\label{al17}
\re\left|v_n\sum_{k=1}^n\left(\re_{k-1}-\re_k\right)\rho_k\xi_k G_k\right|^{2l}\le C_lv_n^{2l}.
\end{align}
By (\ref{al16}) and (\ref{al17}), we conclude that
\begin{align*}
&{\rm P}\left(\max_{u\in S_n'}|nv_nW_1|>\varepsilon\right)
\le C_ln^2\re\left|v_n\sum_{k=1}^n\left(\re_{k-1}-\re_k\right)\rho_k\xi_k G_k\right|^{2l}\le C_ln^{2-3\eta l/49}
\end{align*}
which is summable when $l>49/\eta$. Therefore,
\begin{align}\label{al19}
\max_{u\in S_n'}|W_1|=o_{{\rm a.s.}}(1/(nv_n)).
\end{align}

It is obvious that for $u\in S_n'$
\begin{align*}
&\left|\frac1{p}\rho_k\hat\rho_k\phi_k\rtr\left(\bd_{nk}^{-2}\bS_p\right) G_k\right|^2\\
\le&\left|\frac1{p}\rho_k\hat\rho_k\phi_k\rtr\left(\bd_{nk}^{-2}\bS_p\right) G_k\right|^2I\left(|\rho_k|\le 2C_0 \ {\rm and} \ |\hat\rho_k|\le 3C_0\right)\\
&+\left|\frac1{p}\rho_k\hat\rho_k\phi_k\rtr\left(\bd_{nk}^{-2}\bS_p\right) G_k\right|^2\left[I\left(|\rho_k|>2C_0 \right)+I\left(|\rho_k|\le 2C_0 \ {\rm and} \ |\hat\rho_k|> 3C_0\right)\right]\\
\le&36C_0^4\left|\frac1{p}\phi_k\rtr\left(\bd_{nk}^{-2}\bS_p\right) G_k\right|^2+c_n^{-2}v_n^{-2}\frac{|z|^2}{v_n^2}\left|\phi_k\right|^2\bigg[I\left(|\hat\phi_k|>\frac1{2C_0} \right)\\
&+I\left(|\rho_k|\le 2C_0 \ {\rm and} \ |\rho_k^{-1}+\phi_k|\le\frac1{3C_0}\right)\bigg]\\
\le&36C_0^4 G_k\left|\frac1{p}\rtr\left(\bd_{nk}^{-2}\bS_p\right)\right|^2|\phi_k|^2+Cv_n^{-4}\left|\phi_k\right|^2\bigg[I\left(|\hat\phi_k|\ge\frac1{2C_0} \right)+I\left(|\phi_k|\ge\frac1{6C_0}\right)\bigg].
\end{align*}
By Lemma \ref{le6}, we have for $l\ge1$ and $t>(49-89\eta)l/(95\eta)$
\begin{align*}
&{\rm P}\left(\max_{u\in S_n'}\left|v_n\frac1{p}\sum_{k=1}^n\left(\re_{k-1}-\re_k\right)\rho_k\hat\rho_k\phi_k\rtr\left(\bd_{nk}^{-2}\bS_p\right) G_k\right|>\varepsilon\right)\\
\le&n^2\re\left|\sum_{k=1}^n\left(\re_{k-1}-\re_k\right)v_n\frac1{p}\rho_k\hat\rho_k\phi_k\rtr\left(\bd_{nk}^{-2}\bS_p\right) G_k\right|^{2l}\\
\le&C_ln^2\left[\re\left(\sum_{k=1}^n\re_{k-1}\left|v_n\frac1{p}\rho_k\hat\rho_k\phi_k\rtr\left(\bd_{nk}^{-2}\bS_p\right) G_k\right|^{2}\right)^l+
\sum_{k=1}^n\re\left|v_n\frac1{p}\rho_k\hat\rho_k\phi_k\rtr\left(\bd_{nk}^{-2}\bS_p\right) G_k\right|^{2l}\right]\\
\le&C_ln^2\Bigg[\re\Big(v_n^{2}\sum_{k=1}^n\re_{k-1} G_k\left|\frac1{p}\rtr\left(\bd_{nk}^{-2}\bS_p\right)\right|^2\frac{\rtr\left(\bd_{nk}^{-1}\bS_p(\bd_{nk}^{-1})^*\bS_p\right)}{n^2}\\
&+v_n^{-2}\sum_{k=1}^n\re_{k-1}\left|\phi_k\right|^2\bigg[I\left(|\hat\phi_k|\ge\frac1{2C_0} \right)+I\left(|\phi_k|\ge\frac1{6C_0}\right)\bigg]\Big)^l+v_n^{-6l}
\sum_{k=1}^n\re\left|\phi_k\right|^{2l}\Bigg]\\
\le&C_ln^2\Bigg[\re\left(v_n^{2}\sum_{k=1}^n\re_{k-1} G_k\left|\frac1{p}\rtr\left(\bd_{nk}^{-2}\bS_p\right)\right|^2\frac{\rtr\left(\bd_{nk}^{-1}(\bd_{nk}^{-1})^*\right)}{n^2}\right)^l\\
&+v_n^{-2l}n^{l-1}\sum_{k=1}^n\re_{k-1}\left|\phi_k\right|^{2l}\bigg[I\left(|\hat\phi_k|\ge\frac1{2C_0} \right)+I\left(|\phi_k|\ge\frac1{6C_0}\right)\bigg]+v_n^{-6l}
\sum_{k=1}^n\re\left|\phi_k\right|^{2l}\Bigg]\\
\le&C_ln^2\Bigg[\re\left(\frac1{p}v_n^{2}\sum_{k=1}^n\re_{k-1} G_k\rtr\left(\bd_{nk}^{-2}\bS_p(\bd_{nk}^{-2})^*\bS_p\right)\frac{\rtr\left(\bd_{nk}^{-1}(\bd_{nk}^{-1})^*\right)}{n^2}\right)^l\\
&+v_n^{-2l}n^{l-1}\sum_{k=1}^n\re_{k-1}\left|\phi_k\right|^{2l}\bigg[|\hat\phi_k|^{2t}+|\phi_k|^{2t}\bigg]+v_n^{-6l}n
n^{-2\eta l}v_n^{-2l}\Bigg]\\
\le&C_ln^2\Bigg[\re\left(\frac1{p}v_n^{2}\sum_{k=1}^n\re_{k-1} G_k\rtr\left(\bd_{nk}^{-2}(\bd_{nk}^{-2})^*\right)\frac{\rtr\left(\bd_{nk}^{-1}(\bd_{nk}^{-1})^*\right)}{n^2}\right)^l+v_n^{-2l}n^{l-1}n
n^{-2\eta l-2\eta t}v_n^{-2l-2t}\\
&+v_n^{2l}n^{1-83\eta l/49}\Bigg]\\
\le&C_ln^2\Bigg[\re\left(\frac1{n^2 p}v_n^{2}\sum_{k=1}^n\re_{k-1} G_k\rtr\left(\bd_{nk}^{-2}(\bd_{nk}^{-2})^*\right){\rtr\left(\bd_{nk}^{-1}(\bd_{nk}^{-1})^*\right)}\right)^l+v_n^{2l}n^{-89\eta l/49-95\eta t/49+l}\\
&+v_n^{2l}n^{1-83\eta l/49}\Bigg]\\
\le&C_ln^2\left[\re\left(\frac1{n^2 p}v_n^{2}\sum_{k=1}^n\re_{k-1} G_k\rtr\left(\bd_{nk}^{-2}(\bd_{nk}^{-2})^*\right){\rtr\left(\bd_{nk}^{-1}(\bd_{nk}^{-1})^*\right)}\right)^l+v_n^{2l}\right].
\end{align*}
Using (\ref{al15}), one gets
\begin{align*}
&\sum_{k=1}^n\re_{k-1} G_k\rtr\left(\bd_{nk}^{-2}(\bd_{nk}^{-2})^*\right){\rtr\left(\bd_{nk}^{-1}(\bd_{nk}^{-1})^*\right)}\\
\le&\sum_{k=1}^n\re_{k-1} G_k\sum_{j=1}^p\frac1{\left((\lambda_{kj}-u)^2+v_n^2\right)^2}\sum_{l=1}^p\frac1{(\lambda_{kl}-u)^2+v_n^2}\\
\le&\sum_{k=1}^n\re_{k-1}G_k\left[p\underline\varepsilon^{-4}
+pv_n^{-4}F_{nk}([a',b'])\right]\left[p\underline\varepsilon^{-2}
+pv_n^{-2}F_{nk}([a',b'])\right]\\
\le& Cn^{-3}.
\end{align*}
Thus, for $l\ge1$
\begin{align*}
&{\rm P}\left(\max_{u\in S_n'}\left|v_n\frac1{p}\sum_{k=1}^n\left(\re_{k-1}-\re_k\right)\rho_k\hat\rho_k\phi_k\rtr\left(\bd_{nk}^{-2}\bS_p\right) G_k\right|>\varepsilon\right)
\le C_ln^{2-3\eta l/49}
\end{align*}
which is summable when $l>49/\eta$. Therefore,
\begin{align}\label{al20}
\max_{u\in S_n'}|W_2|=o_{{\rm a.s.}}(1/(nv_n)).
\end{align}
Consequently, (\ref{al18}) follows from (\ref{al19}) and (\ref{al20}).

\subsubsection{Convergence of the Expected Value}
In this step, we are going to show that for $z=u+iv_n=u+in^{-3\eta/98}$,
\begin{align*}
\sup_{u\in[a,b]}|\re\underline m_n(z)-\underline m_n^0(z)|=O(n^{-1}).
\end{align*}

As in \ref{se2}, let
\begin{align*}
w_n'=-\frac1z\int\frac1{1+t\re\underline m_n(z)}dH_n(t)-\re m_n(z)
\end{align*}
and
\begin{align*}
R_n'=-z-\frac1{\re\underline m_n(z)}+c_n\int\frac1{1+t\re\underline m_n(z)}dH_n(t).
\end{align*}
Then $R_n'=w_n'zc_n/\re\underline m_n(z)$ and
\begin{align*}
{\re\underline m_n(z)}=\frac1{-z-R_n'+c_n\int\frac1{1+t\re\underline m_n(z)}dH_n(t)}.
\end{align*}
Together with (\ref{al6}), one finds
\begin{align*}
\re{\underline m_n(z)}-&{\underline m_n^0(z)}=\frac{1}{{c_n}\int\frac t{\re\underline m_n(z)t+1}dH_n(t)-z-R_n'}-\frac{1}{{c_n}\int\frac t{\underline m_n^0(z)t+1}dH_n(t)-z}\\
=&\frac{{c_n}\left({\re\underline m_n(z)}-{\underline m_n^0(z)}\right)\int\frac {t^2}{\left(\re\underline m_n(z)t+1\right)\left(\underline m_n^0(z)t+1\right)}dH_n(t)}{\left({c_n}\int\frac t{\re\underline m_n(z)t+1}dH_n(t)-z-R_n'\right)\left({c_n}\int\frac t{\underline m_n^0(z)t+1}dH_n(t)-z\right)}+{\re\underline m_n(z)}{\underline m_n^0(z)}R_n'.
\end{align*}
Let $\underline m_2^0(z)=\Im \underline m_n^0(z)$. In \cite{Bai1998No}, it has been shown that $\sup_{u\in[a,b]}|\underline m_n^0(z)|$ is bounded in $n$ and
\begin{align*}
\sup_{u\in[a,b]}\frac{m_2^0c_n\int\frac{t^2dH_n(t)}{|1+t\underline m_n^0(z)|^2}}{v_n+{m_2^0c_n\int\frac{t^2dH_n(t)}{|1+t\underline m_n^0(z)|^2}}}
\end{align*}
is bounded away from 1 for all $n$. Therefore, if $|R_n'|\le Cn^{-1}$ is true, we shall get, for all $n$ sufficiently large,
\begin{align*}
\sup_{u\in[a,b]}|\re{\underline m_n(z)}-&{\underline m_n^0(z)}|\le C|{\re\underline m_n(z)}{\underline m_n^0(z)}R_n'|\le Cn^{-1}
\end{align*}
where the last inequality is from $(\ref{al13})$ in \ref{se2}.

Now, we are in position to show $|R_n'|\le Cn^{-1}$. Write
\begin{align*}
\bd_n-\left(-z\re\underline m_n(z)\bS_p-z\bi_p\right)=\sum_{k=1}^n\br_k\br_k^*-\left(-z\re\underline m_n(z)\bS_p\right).
\end{align*}
Taking first inverses and then the expected value, we get from (\ref{al4}) and (\ref{al5}) in \ref{se2}
\begin{align*}
&\left(-z\re\underline m_n(z)\bS_p-z\bi_p\right)^{-1}-\re\bd_n^{-1}\\
=&\left(-z\re\underline m_n(z)\bS_p-z\bi_p\right)^{-1}\re\left[\left(\sum_{k=1}^n\br_k\br_k^*-\left(-z\re\underline m_n(z)\bS_p\right)\right)\bd_n^{-1}\right]\\
=&-z^{-1}\left(\re\underline m_n(z)\bS_p+\bi_p\right)^{-1}\re\left(\sum_{k=1}^n\rho_k\br_k\br_k^*\bd_{nk}^{-1}-(\re\rho_1)\bS_p\bd_n^{-1}\right)\\
=&-z^{-1}\left(\re\underline m_n(z)\bS_p+\bi_p\right)^{-1}\sum_{k=1}^n\re\rho_k\left(\br_k\br_k^*\bd_{nk}^{-1}-\frac1n\bS_p\re\bd_n^{-1}\right)\\
=&-nz^{-1}\re\rho_1\left[\left(\re\underline m_n(z)\bS_p+\bi_p\right)^{-1}\br_1\br_1^*\bd_{n1}^{-1}-\frac1n\left(\re\underline m_n(z)\bS_p+\bi_p\right)^{-1}\bS_p\re\bd_n^{-1}\right].
\end{align*}
Taking the trace on both sides and dividing by $-n/z$, we obtain
\begin{align}\label{al30}
&-c_nzw_n'=c_nz\re m_n(z)+c_n\int\frac1{1+t\re\underline m_n(z)}dH_n(t)\\
=&\re\rho_1\left[\br_1^*\bd_{n1}^{-1}\left(\re\underline m_n(z)\bS_p+\bi_p\right)^{-1}\br_1-\frac1n\rtr\left(\re\underline m_n(z)\bS_p+\bi_p\right)^{-1}\bS_p\re\bd_n^{-1}\right]\notag\\
=&\re\rho_1\left[\br_1^*\bd_{n1}^{-1}\left(\re\underline m_n(z)\bS_p+\bi_p\right)^{-1}\br_1-\frac1n\rtr\left(\re\underline m_n(z)\bS_p+\bi_p\right)^{-1}\bS_p\bd_{n1}^{-1}\right]\notag\\
&+\frac1n\re\rho_1\left[\rtr\left(\re\underline m_n(z)\bS_p+\bi_p\right)^{-1}\bS_p\bd_{n1}^{-1}-\re\rtr\left(\re\underline m_n(z)\bS_p+\bi_p\right)^{-1}\bS_p\bd_{n1}^{-1}\right]\notag\\
&+\frac1n\re\rho_1\left[\re\rtr\left(\re\underline m_n(z)\bS_p+\bi_p\right)^{-1}\bS_p\bd_{n1}^{-1}-\re\rtr\left(\re\underline m_n(z)\bS_p+\bi_p\right)^{-1}\bS_p\bd_n^{-1}\right]\notag\\
\triangleq&T_1+T_2+T_3.\notag
\end{align}

The remaining task is showing the uniformly bound of $c_nzw_n'$ for $u\in[a,b]$. In last section, we have shown that $\sup_{u\in[a,b]}|\re\rho_1|$ and $\sup_{u\in[a,b]}|b_n|$ are bounded. Similarly, we shall show that $\sup_{u\in[a,b]}|\re\rho_{1j}|$ and $\sup_{u\in[a,b]}|b_{1n}|$ are also bounded. From (\ref{al13}) in \ref{se2} and the fact $-1/m_n^0(z)$ stays uniformly away from the eigenvalues of $\bS_p$ for all $u\in[a,b]$, it follows that
\begin{align}\label{al21}
\sup_{u\in[a,b]}\left\|\left(\re\underline m_n(z)\bS_p+\bi_p\right)^{-1}\right\|\le C.
\end{align}

By (\ref{al21}) and the relationship
\begin{align}\label{al29}
\rho_1=b_n-b_n^2\hat\phi_1+b_n^2\rho_1\hat\phi_1^2,
\end{align}
 we get
\begin{align*}
\sup_{u\in[a,b]}|T_1|=&\sup_{u\in[a,b]}|b_n|^2\Bigg|\re\left(\hat\phi_1-\rho_1\hat\phi_1^2\right)\Bigg[\br_1^*\bd_{n1}^{-1}\left(\re\underline m_n(z)\bS_p+\bi_p\right)^{-1}\br_1\\
&\qquad\qquad\qquad-\frac1n\rtr\left(\re\underline m_n(z)\bS_p+\bi_p\right)^{-1}\bS_p\bd_{n1}^{-1}\Bigg]\Bigg|\\
\le&\sup_{u\in[a,b]}\left(\re|\hat\phi_1|^2+\frac{|z|^2}{v_n^2}\re|\hat\phi_1|^4\right)^{1/2}
\Bigg[\re\Bigg|\br_1^*\bd_{n1}^{-1}\left(\re\underline m_n(z)\bS_p+\bi_p\right)^{-1}\br_1\\
&\qquad\qquad\qquad-\frac1n\rtr\left(\re\underline m_n(z)\bS_p+\bi_p\right)^{-1}\bS_p\bd_{n1}^{-1}\Bigg|^2\Bigg]^{1/2}\\
\le&\frac Cn\sup_{u\in[a,b]}\left(\re|\hat\phi_1|^2+\frac{|z|^2}{v_n^2}\re|\hat\phi_1|^4\right)^{1/2}
\Bigg[\re\rtr\bd_{n1}^{-1}\left(\bd_{n1}^*\right)^{-1}\Bigg]^{1/2}.
\end{align*}
Using (\ref{al15}), we have for $k\ge1$ and $l=1,2$,
\begin{align}\label{al24}
&\sup_{u\in[a,b]}\re\left[\rtr\bd_{n1}^{-l}\left(\bd_{n1}^*\right)^{-l}\right]^k=\sup_{u\in[a,b]}\re\left[\sum_{j=1}^p\left(\frac1{(\lambda_{1j}-u)^2+v_n^2}\right)^l\right]^k\\
=&\sup_{u\in[a,b]}\re\left[\sum_{j\notin[a',b']}\left(\frac1{(\lambda_{1j}-u)^2+v_n^2}\right)^l+\sum_{j\in[a',b']}\left(\frac1{(\lambda_{1j}-u)^2+v_n^2}\right)^l\right]^k\notag\\
\le&\sup_{u\in[a,b]}\re\left[p\underline\varepsilon^{-2l}+pv_n^{-2l}F_{n1}([a',b'])\right]^k\le Cn^k.\notag
\end{align}
Likewise, it can be verified that for $k\ge1$ and $l=1,2$,
\begin{align}\label{al22}
&\sup_{u\in[a,b]}\re\left[\rtr\bd_{n12}^{-l}\left(\bd_{n12}^*\right)^{-l}\right]^k\le Cn^k.
\end{align}
By (\ref{al4}), (\ref{al23}), (\ref{al22}), Corollary \ref{co1}, and $\rho_{1j}=b_{1n}-\rho_{1j}b_{1n}\hat\phi_{1j}$, we have for any nonrandom $p\times p$ matrix $\ba$ with bounded norm,
\begin{align}\label{al28}
&\sup_{u\in[a,b]}\re\left|\rtr\ba\bd_{n1}^{-1}-\re\rtr\ba\bd_{n1}^{-1}\right|^2=\sup_{u\in[a,b]}\sum_{j=2}^n\re\left|(\re_j-\re_{j-1})\rtr\ba\bd_{n1}^{-1}\right|^2\\
\le&2\sup_{u\in[a,b]}\sum_{j=2}^n\re\left|\rho_{1j}\br_j^*\bd_{n1j}^{-1}\ba\bd_{n1j}^{-1}\br_j\right|^2\notag\\
=&2(n-1)\sup_{u\in[a,b]}\re\left|\left(b_{1n}-\rho_{12}b_{1n}\hat\phi_{12}\right)\br_2^*\bd_{n12}^{-1}\ba\bd_{n12}^{-1}\br_2\right|^2\notag\\
\le&Cn\sup_{u\in[a,b]}\left[\re\left|\br_2^*\bd_{n12}^{-1}\ba\bd_{n12}^{-1}\br_2\right|^2
+v_n^{-2}\left(\re\left|\hat\phi_{12}\right|^4\re\left|\br_2^*\bd_{n12}^{-1}\ba\bd_{n12}^{-1}\br_2\right|^4\right)^{1/2}\right]\notag\\
\le&Cn\sup_{u\in[a,b]}\left[\re\left|\br_2^*\bd_{n12}^{-1}(\bd_{n12}^*)^{-1}\br_2\right|^2
+n^{-2\eta}v_n^{-4}\left(\re\left|\br_2^*\bd_{n12}^{-1}(\bd_{n12}^*)^{-1}\br_2\right|^4\right)^{1/2}\right]\notag\\
\le&Cn\sup_{u\in[a,b]}\Bigg[n^{-2}\left(\rtr\bd_{n12}^{-1}(\bd_{n12}^*)^{-1}\right)^2+n^{-2}n\left\|\bd_{n12}^{-1}(\bd_{n12}^*)^{-1}\right\|^2\notag\\
&+n^{-2\eta}v_n^{-4}\left(n^{-4}\left(\rtr\bd_{n12}^{-1}(\bd_{n12}^*)^{-1}\right)^4+n^{-4}n^{2}\left\|\bd_{n12}^{-1}(\bd_{n12}^*)^{-1}\right\|^4\right)^{1/2}\Bigg]\notag\\
\le&Cn\sup_{u\in[a,b]}\Bigg[1+n^{-1}v_n^{-4}+n^{-2\eta}v_n^{-4}\left(1+n^{-2}v_n^{-8}\right)^{1/2}\Bigg]\le Cn.\notag
\end{align}
From the above inequality, one obtains
\begin{align*}
&\re\left|\hat\phi_1-\phi_1\right|^{2}
=n^{-2}\re\left|\rtr\left(\bS_p\bd_{n1}^{-1}\right)-\re\rtr\left(\bS_p\bd_{n1}^{-1}\right)\right|^{2}
\le Cn^{-1}.
\end{align*}
Together with
\begin{align*}
\sup_{u\in[a,b]}\re|\phi_1|^2\le\sup_{u\in[a,b]}Cn^{-2}\re\rtr\bd_{n1}^{-1}(\bd_{n1}^*)^{-1}\le Cn^{-1},
\end{align*}
we get
\begin{align}\label{al25}
\sup_{u\in[a,b]}\re|\hat\phi_1|^2\le2\sup_{u\in[a,b]}\re|\hat\phi_1-\phi_1|^2+2\sup_{u\in[a,b]}\re|\phi_1|^2\le Cn^{-1}.
\end{align}
Applying Corollary \ref{co1} and (\ref{al24}), it implies
\begin{align*}
\sup_{u\in[a,b]}\re|\phi_1|^4\le Cn^{-4}\sup_{u\in[a,b]}n^{2}\re\left\|\bd_{n1}^{-1}\right\|^2\le Cn^{-2}v_n^{-2}.
\end{align*}
Combining (\ref{al26}) and the above inequality, we get
\begin{align}\label{al27}
\sup_{u\in[a,b]}\re|\hat\phi_1|^4\le8\sup_{u\in[a,b]}\re|\hat\phi_1-\phi_1|^4+8\sup_{u\in[a,b]}\re|\phi_1|^4\le Cn^{-2}v_n^{-4}.
\end{align}
From (\ref{al24}), (\ref{al25}), and (\ref{al27}), we conclude that
\begin{align*}
\sup_{u\in[a,b]}|T_1|
\le&\frac Cn\sup_{u\in[a,b]}\left(n^{-1}+\frac{|z|^2}{v_n^2}n^{-2}v_n^{-4}\right)^{1/2}n^{1/2}\le Cn^{-1}.
\end{align*}

By  (\ref{al29}) and (\ref{al28})-(\ref{al27}), it follows
\begin{align*}
\sup_{u\in[a,b]}|T_2|
=&\sup_{u\in[a,b]}\Bigg|\frac1n\re(b_n^2\hat\phi_1-b_n^2\rho_1\hat\phi_1^2)\Big[\rtr\left(\re\underline m_n(z)\bS_p+\bi_p\right)^{-1}\bS_p\bd_{n1}^{-1}\\
&-\re\rtr\left(\re\underline m_n(z)\bS_p+\bi_p\right)^{-1}\bS_p\bd_{n1}^{-1}\Big]\Bigg|\\
\le&\sup_{u\in[a,b]}\frac Cn\left(\re\left|\rtr\left(\re\underline m_n(z)\bS_p+\bi_p\right)^{-1}\bS_p\bd_{n1}^{-1}-\re\rtr\left(\re\underline m_n(z)\bS_p+\bi_p\right)^{-1}\bS_p\bd_{n1}^{-1}\right|^2\right)^{1/2}\\
&\times\left(\re|\hat\phi_1|^2+v_n^{-2}\re|\hat\phi_1|^4\right)^{1/2}\\
\le&\frac C{ n}\sqrt n\left(n^{-1}+n^{-2}v_n^{-6}\right)^{1/2}\le Cn.
\end{align*}

Using (\ref{al4}), (\ref{al24}), (\ref{al25}), Corollary \ref{co1}, and $\rho_1=b_n-\rho_1b_n\hat\phi_1$
\begin{align*}
\sup_{u\in[a,b]}|T_3|
\le&\frac Cn\sup_{u\in[a,b]}\Bigg|\re\rtr\left(\re\underline m_n(z)\bS_p+\bi_p\right)^{-1}\bS_p\bd_{n1}^{-1}-\re\rtr\left(\re\underline m_n(z)\bS_p+\bi_p\right)^{-1}\bS_p\bd_n^{-1}\Bigg|\\
=&\frac Cn\sup_{u\in[a,b]}\Bigg|\re\rho_1\br_1^*\bd_{n1}^{-1}\left(\re\underline m_n(z)\bS_p+\bi_p\right)^{-1}\bS_p\bd_{n1}^{-1}\br_1\Bigg|\\
=&\frac Cn\sup_{u\in[a,b]}\Bigg|\re(b_n-\rho_1b_n\hat\phi_1)\br_1^*\bd_{n1}^{-1}\left(\re\underline m_n(z)\bS_p+\bi_p\right)^{-1}\bS_p\bd_{n1}^{-1}\br_1\Bigg|\\
\le&\frac C{n^2}\sup_{u\in[a,b]}\Bigg|\re\rtr\bd_{n1}^{-1}\left(\re\underline m_n(z)\bS_p+\bi_p\right)^{-1}\bS_p\bd_{n1}^{-1}\bS_p\Bigg|\\
&+\frac Cn\sup_{u\in[a,b]}v_n^{-1}\re^{1/2}|\hat\phi_1|^2\re^{1/2}\left|\br_1^*\bd_{n1}^{-1}\left(\re\underline m_n(z)\bS_p+\bi_p\right)^{-1}\bS_p\bd_{n1}^{-1}\br_1\right|^2\\
\le&\frac C{n^2}\sup_{u\in[a,b]}\left|\re\rtr\bd_{n1}^{-1}(\bd_{n1}^*)^{-1}\right|+\frac C{ n^{3/2}}\sup_{u\in[a,b]}v_n^{-1}\re^{1/2}\left|\br_1^*\bd_{n1}^{-1}(\bd_{n1}^*)^{-1}\br_1\right|^2\\
\le&\frac C{n}+\frac C{ n^{3/2}}\sup_{u\in[a,b]}v_n^{-1}\left[\frac{n\re\left\|\bd_{n1}^{-1}(\bd_{n1}^*)^{-1}\right\|^2}{n^2}\right]^{1/2}\\
\le&\frac C{n}+\frac C{ n^{3/2}}v_n^{-1}\le Cn^{-1}.
\end{align*}

From the above three inequalities and (\ref{al30}), we have
\begin{align*}
\sup_{u\in[a,b]}|w_n'|\le Cn^{-1}.
\end{align*}
From $(\ref{al13}$), we see that $\re\underline m_n(z)$ must be uniformly bounded away from $0$ for all $u\in[a,b]$ and all $n$. Therefore,
\begin{align*}
\sup_{u\in[a,b]}|R_n'|=\sup_{u\in[a,b]}\left|w_n'zc_n/\re\underline m_n(z)\right|\le Cn^{-1}.
\end{align*}

\subsubsection{Completing the proof of theorem \ref{tms}}

The results presented in the last two subsections implies that for $z=u+iv_n=u+in^{-3\eta/98}$,
\begin{align*}
\sup_{u\in[a,b]}|\underline m_n(z)-\underline m_n^0(z)|=o((nv_n)^{-1}) a.s..
\end{align*}
Following
the same arguments as \cite{Bai1998No}, one can easily prove that
\begin{align*}
\sup_{u\in[a,b]}|\int\frac{I_{[a,b]^c}d(F^{\underline{\mb S_n}(\lambda)}-F^{c_n,H_n(\lambda)})}{\(\(u-\lambda_j\)^2+v_n^2\)\(\(u-\lambda_j\)^2+2v_n^2\)\cdots\(\(u-\lambda_j\)^2+C v_n^2\)}+\\ \sum_{\lambda_{j}\in[a',b']}\frac{v_n^{{2C}}}{\(\(u-\lambda_j\)^2+v_n^2\)\(\(u-\lambda_j\)^2+2v_n^2\)\cdots\(\(u-\lambda_j\)^2+C v_n^2\)}|=o(1),\ a.s.
\end{align*}
where $\lambda_j'$s denotes the eigenvalues of $\underline{\mb S_n}$ and $C$ is a constant (even number) determined by $\eta$. From this, combining with the fact that the integral
converges a.s. to 0, one can argue that, with probability one, no eigenvalue of $\mb S_n$ appears in $[a,b]$
for all sufficiently large $n$.

The proof of this theorem is complete by the argument above and subsection \ref{truncation}.
\subsection{Proof of theorem \ref{tmt}}
When $\mb \Sigma_p=\mb I_p$, by theorem \ref{tmf}, we have
$$z=-\frac{1}{\underline m(z)}+\frac{c}{1+\underline m(z)},$$
which implies
$$\underline m(z)=\frac{-(z+1-c)+\sqrt{(z-1-y)^2-4y}}{2z}.$$
Then the LSD of $\mb S_n$ is the standard M-P law, which is supported on $[(1-\sqrt{c})^2,(1+\sqrt{c})^2]$. The proof of this theorem is complete by combining the result of theorem \ref{tms}.

\section{List of Auxiliary lemmas}

\subsection{A key lemma that need to prove}

\begin{lemma}\label{le3}
Let $\ba=(a_{jk})$ be a $p\times p$ nonrandom matrix and $\mb x=(x_1,\cdots,x_m)'$ be a random vector of independent entries. Assume that $\re x_j=0$, $\re|x_j|^2=1$,  $\sup_{j}\re|x_j|^{6+\delta}\le M$, and $|x_j|\le n^{1/2-\eta}/b_j$, $p/n\to c\in(0,\infty)$. Here $b_j$ is defined in Section 2. Then for any $l\ge1$, as $n\to \infty$
\begin{align*}
\re|\mb x^*\bb_n^*\ba\bb_n\mb x-\rtr\ba\bS_p|^l\le C_l n^{(1-\eta)l}\|\ba\|^l
\end{align*}
where $C_l$ is a constant depending on $l$ only and ${\bf \Sigma_p}=\bb_n\bb_n^*$.
\end{lemma}
\begin{proof}
Let $\bh=(h_{jk})=\bb^*_n\ba\bb_n$, we have
\begin{align*}
\mb x^*\bh\mb x-\rtr\bh=\sum_{j=1}^mh_{jj}\left(|x_j|^2-1\right)+\sum_{j=1}^m\sum_{k=1}^{j-1}\left(h_{kj}\bar x_kx_j+h_{jk}\bar x_jx_k\right).
\end{align*}
At first, we deduce
\begin{align*}
|h_{jk}|=\left|\be_j'\bb^*_n\ba\bb_n\be_k\right|\le\|\ba\|\sqrt{\be_j'\bb^*_n\bb_n\be_j}\sqrt{\be_k'\bb^*_n\bb_n\be_k}=b_jb_k\|\ba\|
\end{align*}
where $\be_j$ is a vector with the $j$-th element $1$ and the remaining elements zero.

Now, assume $1<l\le 2$. By Lemma \ref{le4} and Lemma \ref{le5}, we have
\begin{align*}
&\re\left|\sum_{j=1}^mh_{jj}\left(|x_j|^2-1\right)\right|^l\le C\re\left[\sum_{j=1}^m|h_{jj}|^2\left(|x_j|^2-1\right)^2\right]^{l/2}\\
\le& C\sum_{j=1}^m|h_{jj}|^l\re\left||x_j|^2-1\right|^{l}\le C\sum_{j=1}^mb_j^{2l}\|\ba\|^{l}\le C\sum_{j=1}^mb_j^{2}\|\ba\|^l\le Cn\|\ba\|^l.
\end{align*}
Furthermore, by the Holder inequality,
\begin{align*}
&\re\left|\sum_{j=1}^m\sum_{k=1}^{j-1}\left(h_{kj}\bar x_kx_j+h_{jk}\bar x_jx_k\right)\right|^l
\le C\left[\re\left|\sum_{j=1}^m\sum_{k=1}^{j-1}\left(h_{kj}\bar x_kx_j+h_{jk}\bar x_jx_k\right)\right|^2\right]^{l/2}\\
\le&C\left[\sum_{j=1}^m\sum_{k=1}^{j-1}\left(|h_{kj}|^2+|h_{jk}|^2\right)\right]^{l/2}\le C\left[\rtr\bh\bh^*\right]^{l/2}\le Cn^{l/2}\|\ba\|^l\le Cn\|\ba\|^l.
\end{align*}
Combining the two inequalities above, we obtain for $1<l\le 2$
\begin{align}\label{al10}
\re|\mb x^*\bh\mb x-\rtr\bh\bh^*|^l\le Cn\|\ba\|^l.
\end{align}
which implies that
\begin{align*}
\re|\mb x^*\bh\mb x-\rtr\bh\bh^*|\le\re^{1/2}|\mb x^*\bh\mb x-\rtr\bh\bh^*|^2\le C\sqrt n\|\ba\|.
\end{align*}

We shall proceed with the proof of the lemma by induction on $l$. And consider the case $2<l\le4$. Using Lemma \ref{le6} and Lemma \ref{le5},
\begin{align*}
&\re\left|\sum_{j=1}^mh_{jj}\left(|x_j|^2-1\right)\right|^l\\
\le &C\left[\left(\sum_{j=1}^m|h_{jj}|^2\re\left(|x_j|^2-1\right)^2\right)^{{l/2}}+\sum_{j=1}^m|h_{jj}|^l\re\left(|x_j|^2-1\right)^l\right]\\
\le &C\left[\left(\rtr\bh\bh^*\right)^{{l/2}}+\sum_{j=1}^mb_j^{2l}\|\ba\|^l\re|x_j|^{2l}\right]\le C\left[n^{l/2}\|\ba\|^{{l}}+\sum_{j=1}^mb_j^{2l}\|\ba\|^l\frac{n^{2(1/2-\eta)}}{b_j^{2}}\re|x_j|^{2l-2}\right]\\
\le &C\left[n^{l/2}\|\ba\|^{{l}}+\sum_{j=1}^mb_j^{2}\|\ba\|^l{n^{1-2\eta }}\right]\le C\left[n^{l/2}\|\ba\|^{{l}}+n^{2-2\eta }\|\ba\|^l\right].
\end{align*}
For the same reason, with notation $\re_j(\cdot)$ for the conditional expectation given $\{x_1,\cdots,x_j\}$, we have
\begin{align*}
&\re\left|\sum_{j=1}^m\sum_{k=1}^{j-1}h_{kj}\bar x_kx_j\right|^l\\
\le &C\left[\re\left(\sum_{j=1}^m\re_{j-1}\left|\sum_{k=1}^{j-1}h_{kj}\bar x_kx_j\right|^2\right)^{l/2}+\sum_{j=1}^m\re\left|\sum_{k=1}^{j-1}h_{kj}\bar x_kx_j\right|^l\right]\\
\le &C\left[\re\left(\sum_{j=1}^m\left|\sum_{k=1}^{j-1}h_{kj}\bar x_k\right|^2\right)^{l/2}+\sum_{j=1}^m\re\left|\sum_{k=1}^{j-1}h_{kj}\bar x_k\right|^l\right]\\
\le &C\left[\re\left(\sum_{j=1}^m\left|\re_{j-1}\sum_{k=1}^{n}h_{kj}\bar x_k\right|^2\right)^{l/2}+\sum_{j=1}^m\left(\sum_{k=1}^{j-1}|h_{kj}|^2\right)^{l/2}
+\sum_{j=1}^m\sum_{k=1}^{j-1}\re\left|h_{kj}\right|^l\right]\\
\le &C\left[\re\left(\sum_{j=1}^m\left|\sum_{k=1}^{n}h_{kj}\bar x_k\right|^2\right)^{l/2}+\sum_{j=1}^m\left((\bh^*\bh)_{jj}\right)^{l/2}
+\rtr\left(\bh^*\bh\right)^{l/2}\right]\\
\le &C\left[\re\left(\bx^*\bh\bh^*\bx\right)^{l/2}+\rtr\left(\bh^*\bh\right)^{l/2}\right]\le Cn\|\ba\|^l
\end{align*}
The last inequality is from (\ref{al10}) with $\bh$ replaced by $\bh\bh^*$.
Together with the two inequalities above, we conclude for $2<l\le 4$
\begin{align*}
\re|\bx^*\bh\bx-\rtr\bh\bh^*|^l\le C\left[n^{l/2}\|\ba\|^{{l}}+n^{2-2\eta }\|\ba\|^l\right]\le Cn^{(1-\eta)l}\|\ba\|^l.
\end{align*}

In the following, consider the case $2^{\theta}<l\le2^{\theta+1}$ with $\theta\ge2$. Likewise, using Lemma \ref{le6} and Lemma \ref{le5}, we deduce
\begin{align*}
&\re\left|\sum_{j=1}^mh_{jj}\left(|x_j|^2-1\right)\right|^l\\
\le &C_l\left[\left(\sum_{j=1}^m|h_{jj}|^2\re\left(|x_j|^2-1\right)^2\right)^{{l/2}}+\sum_{j=1}^m|h_{jj}|^l\re\left(|x_j|^2-1\right)^l\right]\\
\le &C_l\left[\left(\rtr\bh\bh^*\right)^{{l/2}}+\sum_{j=1}^mb_j^{2l}\|\ba\|^l\re|x_j|^{2l}\right]\le C_l\left[n^{l/2}\|\ba\|^{{l}}+\sum_{j=1}^mb_j^{2l}\|\ba\|^l\frac{n^{(1/2-\eta)(2l-6)}}{b_j^{2l-6}}\right]\\
\le &C_l\left[n^{l/2}\|\ba\|^{{l}}+\sum_{j=1}^mb_j^{6}\|\ba\|^l{n^{l-3-2\eta l}}\right]\le C\left[n^{l/2}\|\ba\|^{{l}}+n\|\ba\|^l{n^{l-3-2\eta l}}\right]
\le C_ln^{(1-\eta)l}\|\ba\|^l.
\end{align*}
and
\begin{align*}
&\re\left|\sum_{j=1}^m\sum_{k=1}^{j-1}h_{kj}\bar x_kx_j\right|^l\\
\le &C_l\left[\re\left(\sum_{j=1}^m\re_{j-1}\left|\sum_{k=1}^{j-1}h_{kj}\bar x_kx_j\right|^2\right)^{l/2}+\sum_{j=1}^m\re\left|\sum_{k=1}^{j-1}h_{kj}\bar x_kx_j\right|^l\right]\\
\le &C_l\left[\re\left(\sum_{j=1}^m\left|\sum_{k=1}^{j-1}h_{kj}\bar x_k\right|^2\right)^{l/2}+\sum_{j=1}^m\re|x_j|^l\re\left|\sum_{k=1}^{j-1}h_{kj}\bar x_k\right|^l\right]\\
\le &C_l\left[\re\left(\sum_{j=1}^m\left|\re_{j-1}\sum_{k=1}^{n}h_{kj}\bar x_k\right|^2\right)^{l/2}+\sum_{j=1}^m\re|x_j|^l\left(\sum_{k=1}^{j-1}|h_{kj}|^2\right)^{l/2}
+\sum_{j=1}^m\re|x_j|^l\sum_{k=1}^{j-1}\left|h_{kj}\right|^l\re|x_k|^l\right]\\
\le &C_l\left[\re\left(\sum_{j=1}^m\left|\sum_{k=1}^{n}h_{kj}\bar x_k\right|^2\right)^{l/2}+\sum_{j=1}^m\re|x_j|^lb_j^l\left(\sum_{k=1}^{m}b_{k}^2\right)^{l/2}\|\ba\|^l
+\sum_{j=1}^m\re|x_j|^{l}b_j^l\sum_{k=1}^{j-1}\re|x_k|^{l}b_k^l\|\ba\|^l\right]\\
\le &C_l\left[\re\left(\sum_{j=1}^m\left|\sum_{k=1}^{n}h_{kj}\bar x_k\right|^2\right)^{l/2}+n^{l/2}\sum_{j=1}^m\re|x_j|^lb_j^l\|\ba\|^l
+\left(\sum_{j=1}^m\re|x_j|^{l}b_j^l\right)^2\|\ba\|^l\right]\\
\le &C_l\left[\re\left(\sum_{j=1}^m\left|\sum_{k=1}^{n}h_{kj}\bar x_k\right|^2\right)^{l/2}+n^{l/2}\sum_{j=1}^m\frac{n^{(1/2-\eta)(l-4)}}{b_j^{l-4}}b_j^l\|\ba\|^l
+\left(\sum_{j=1}^m\frac{n^{(1/2-\eta)(l-4)}}{b_j^{l-4}}b_j^l\right)^2\|\ba\|^l\right]\\
\le &C_l\left[\re\left(\sum_{j=1}^m\left|\sum_{k=1}^{n}h_{kj}\bar x_k\right|^2\right)^{l/2}+n^{l/2}\sum_{j=1}^m{n^{l/2-2-\eta l+4\eta}}b_j^4\|\ba\|^l
+\left(\sum_{j=1}^m{n^{l/2-2-\eta l+4\eta}}b_j^4\right)^2\|\ba\|^l\right]\\
\le &C_l\left[\re\left(\mb x^*\bh\bh^*\mb x\right)^{l/2}+{n^{l-\eta l}}\|\ba\|^l
+{n^{l-2\eta l}}\|\ba\|^l\right]\le C_l\left[\re\left(\mb x^*\bh\bh^*\mb x\right)^{l/2}+{n^{l-\eta l}}\|\ba\|^l\right].
\end{align*}
Using the induction hypothesis with $\bh$ replaced by $\bh\bh^*$, it follows that
\begin{align*}
\re\left(\mb x^*\bh\bh^*\mb x\right)^{l/2}\le &C_l\left[\re\left|\mb x^*\bh\bh^*\mb x-\rtr\bh\bh^*\right|^{l/2}+\left(\rtr\bh\bh^*\right)^{l/2}\right]\\
\le&C_l\left[n^{(1-\eta)l/2}\|\ba\|^l+n^{l/2}\|\ba\|^{l}\right]\le C_ln^{(1-\eta)l/2}\|\ba\|^l.
\end{align*}
Consequently, we get
\begin{align*}
\re|\mb x^*\bh\mb x-\rtr\bh\bh^*|^l\le Cn^{l/2}\|\ba\|^l\le Cn^{(1-\eta)l}\|\ba\|^l.
\end{align*}
\end{proof}
From the proof of the above lemma, it is straightforward to show
\begin{col}\label{co1}
Under the conditions of Lemma \ref{le3}, we have
\begin{align*}
\re|\mb x^*\bb_n^*\ba\bb_n\mb x-\rtr\ba\bS_p|^2\le Cn\|\ba\|^2.
\end{align*}
and
\begin{align*}
\re|\mb x^*\bb_n^*\ba\bb_n\mb x-\rtr\ba\bS_p|^4\le Cn^2\|\ba\|^4.
\end{align*}
\end{col}

\subsection{Some Existing Lemmas}

\begin{lemma}[Corollary 7.3.8 of \cite{h1985}]\label{le1}
For $r\times s$ matrices $\ba$ and $\bb$ with respective singular values $\sigma_1\ge\sigma_2\ge\cdots\ge\sigma_q,\tau_1\ge\tau_2\ge\cdots\ge\tau_q$, where $q=\min(r,s)$, we have
\begin{align*}
|\sigma_k-\tau_k|\le\|\bb-\ba\|\quad{\rm for \  all} \  k=1,2,\cdots,q.
\end{align*}
\end{lemma}

\begin{lemma}[Lemma 2.6 of \cite{silverstein1995empirical}]\label{le2}
For $z=u+iv\in\mathbb{C}^+$, let $m_1(z),m_2(z)$ be Stieltjes transforms of any two p.d.f.'s, $\ba$ and $\bb$ $n\times n$ with $\ba$ Hermitian nonnegative definite and $\br\in\mathbb{C}^n$. Then
\begin{align*}
(a)&\quad \|(m_1(z)\ba+\bi_n)^{-1}\|\le\max(4\|\ba\|/v,2),\\
(b)&\quad |\br^*\bb(m_1(z)\ba+\bi_n)^{-1}\br-\br^*\bb(m_1(z)\ba+\bi_n)^{-1}\br|\\
&\quad \quad \quad \le|m_2(z)-m_1(z)| \ \|r\|_2^2 \ \|\bb\| \ \|\ba\| \ \max(4\|\ba\|/v,2)^2\notag
\end{align*}
where $\|\br\|_2$ denotes the Euclidean norm on $\br$.
\end{lemma}

\begin{lemma}[\cite{burkholder1973distribution}]\label{le4}
Let $\{X_k\}$ be a complex martingale difference sequence with respect to the increasing $\sigma$-field $\{\mathcal{F}_k\}$. Then, for $l>1$,
\begin{align*}
\re\left|\sum X_k\right|^l\le C_l\re\left(\sum|X_k|^2\right)^{l/2}.
\end{align*}
\end{lemma}

\begin{lemma}[ \cite{burkholder1973distribution}]\label{le6}
Let $\{X_k\}$ be a complex martingale difference sequence with respect to the increasing $\sigma$-field $\{\mathcal{F}_k\}$. Then, for $l\ge2$,
\begin{align*}
\re\left|\sum X_k\right|^l\le C_l\left[\re\left(\re\left(\sum|X_k|^2|\mathcal{F}_{k-1}\right)^{l/2}\right)+\sum\re|X_k|^l\right].
\end{align*}
\end{lemma}

\begin{lemma}[(3.3.41) of  \cite{h1991}]\label{le5}
For $n\times n$ Hermitian $\ba=(a_{jk})$ with eigenvalues $\lambda_1,\cdots,\lambda_n$, and convex function $f(\cdot)$, we have
\begin{align*}
\sum_{j=1}^nf(a_{jj})\le\sum_{j=1}^nf(\lambda_j).
\end{align*}
\end{lemma}

\begin{lemma}[Corollary 2.1 of  \cite{hall1980martingale}]\label{le7}
If $\{X_j,\mathcal{F}_j,1\le j\le n\}$ is a martingale, then for each $l\ge1$ and $\alpha>0$,
\begin{align*}
\alpha^l{\rm P}\left(\max_{j\le n}|X_j|>\alpha\right)\le\re|X_n|^l.
\end{align*}
\end{lemma}

\begin{lemma}[Lemma 6.8 of \cite{bai2010spectral}]\label{le8}
If, for all $t>0$, $t^l{\rm P}(|X|>t)\le C$ for some positive $l$, then for any positive $q<l$,
\begin{align*}
\re|X|^q\le C^{q/l}\left(\frac l{l-q}\right).
\end{align*}
\end{lemma}

\begin{lemma}[Lemma 2.4 of  \cite{silverstein1995empirical}]\label{le9}
For $n\times n$ Hermitian $\ba$ and $\bb$,
\begin{align*}
\|F^{\ba}-F^{\bb}\|_{KS}\le\frac1n{\rm rank}(\ba-\bb),
\end{align*}
here $\|\cdot\|_{KS}$ denoting the sup norm on functions.
\end{lemma}

\newpage

\begin{supplement}[id=se2]
 \sname{Supplement A}
\stitle{a convergence rate of ${\underline m}_n(z)$}
\sdescription{
Let $\re_0(\cdot)$ denote expectation and $\re_k(\cdot)$ denote conditional expectation with respect to the $\sigma$-field generated by $\br_1,\cdots,\br_k$, we shall show that for $z=u+iv_n=u+in^{-6\eta/49}$ and $r\ge1$,
\begin{align*}
\max_{k\le n}\re_k\left(v_n^{-r}\sup_{u\in\mathbb{R}}\left|\underline m_n(z)-\underline m_n^0(z)\right|^r\right)\xrightarrow{a.s.}0.
\end{align*}
\\
To begin with, we deduce three equalities in order to obtaining the expression of $\underline m_n(z)-\underline m_n^0(z)$. Write
\begin{align*}
\bd_n+z\bi_p=\sum_{j=1}^n\br_j\br_j^*.
\end{align*}
Taking the inverse of $\bd_n$ on the both sides and using (\ref{al4}), we have
\begin{align*}
\bi_p+z\bd_n^{-1}=\sum_{j=1}^n{\br_j\br_j^*}\bd_n^{-1}=\sum_{j=1}^n\rho_j{\br_j\br_j^*}\bd_{nj}^{-1}.
\end{align*}
Then, we deduce by taking the trace on both sides and dividing by $n$,
\begin{align*}
c_n+c_nzm_n(z)=\frac1n\sum_{j=1}^n\rho_j{\br_j^*}\bd_{nj}^{-1}\br_j=1-\frac1n\sum_{j=1}^n\rho_j.
\end{align*}
Together with $(\ref{al2})$, one gets
\begin{align}\label{al5}
\underline m_n(z)=-\frac1{zn}\sum_{j=1}^n\rho_j.
\end{align}
\\
Write
\begin{align*}
\bd_n-\left(-z\underline m_n(z)\bS_p-z\bi_p\right)=\sum_{j=1}^n\br_j\br_j^*-(-z\underline m_n(z))\bS_p.
\end{align*}
Taking inverses and using (\ref{al4}) and (\ref{al5}), we have
\begin{align*}
\left(-z\underline m_n(z)\bS_p-z\bi_p\right)^{-1}&-\bd_n^{-1}=-z^{-1}\left(\underline m_n(z)\bS_p+\bi_p\right)^{-1}\left[\sum_{j=1}^n\br_j\br_j^*-(-z\underline m_n(z))\bS_p\right]\bd_n^{-1}\\
=&-z^{-1}\left(\underline m_n(z)\bS_p+\bi_p\right)^{-1}\left[\sum_{j=1}^n\rho_j\br_j\br_j^*\bd_{nj}^{-1}-\frac1n\sum_{j=1}^n\rho_j\bS_p\bd_n^{-1}\right]\\
=&-z^{-1}\sum_{j=1}^n\rho_j\left(\underline m_n(z)\bS_p+\bi_p\right)^{-1}\left[\br_j\br_j^*\bd_{nj}^{-1}-\frac1n\bS_p\bd_n^{-1}\right].
\end{align*}
Taking the trace and dividing by $p$, we see
\begin{align*}
w_n=&\frac1p\rtr\left(-z\underline m_n(z)\bS_p-z\bi_p\right)^{-1}-m_n(z)\\
=&-\frac1{p z}\sum_{j=1}^n\rho_j\Bigg\{\br_j^*\bd_{nj}^{-1}\left(\underline m_n(z)\bS_p+\bi_p\right)^{-1}\br_j
-\frac1n\rtr\left[\left(\underline m_n(z)\bS_p+\bi_p\right)^{-1}\bS_p\bd_n^{-1}\right]\Bigg\}\\
\triangleq&-\frac1{p z}\sum_{j=1}^n\rho_jd_j.
\end{align*}
\\
Rewriting $w_n$ in terms of $\underline m_n(z)$, it follows that by (\ref{al2})
\begin{align*}
w_n=&-\frac1{z}\int\frac1{\underline m_n(z)t+1}dH_n(t)-\frac{z\underline m_n(z)+1-c_n}{c_nz}\\
=&\frac{\underline m_n(z)}{c_nz}\left[{c_n}\int\frac t{\underline m_n(z)t+1}dH_n(t)-z-\frac{1}{\underline m_n(z)}\right]\\
\triangleq&\frac{\underline m_n(z)}{c_nz} R_n
\end{align*}
which yields $R_n={c_nzw_n}/{\underline m_n(z)}$ and
\begin{align}\label{al7}
{\underline m_n(z)}=&\frac{1}{{c_n}\int\frac t{\underline m_n(z)t+1}dH_n(t)-z-R_n}.
\end{align}
Consequently, we obtain from the above equality and (\ref{al6}),
\begin{align*}
{\underline m_n(z)}-&{\underline m_n^0(z)}=\frac{1}{{c_n}\int\frac t{\underline m_n(z)t+1}dH_n(t)-z-R_n}-\frac{1}{{c_n}\int\frac t{\underline m_n^0(z)t+1}dH_n(t)-z}\\
=&\frac{{c_n}\left({\underline m_n(z)}-{\underline m_n^0(z)}\right)\int\frac {t^2}{\left(\underline m_n(z)t+1\right)\left(\underline m_n^0(z)t+1\right)}dH_n(t)}{\left({c_n}\int\frac t{\underline m_n(z)t+1}dH_n(t)-z-R_n\right)\left({c_n}\int\frac t{\underline m_n^0(z)t+1}dH_n(t)-z\right)}+{\underline m_n(z)}{\underline m_n^0(z)}R_n.
\end{align*}
\\
Let $\underline m_{n2}(z)=\Im \underline m_n(z)$, then from (\ref{al6}) and (\ref{al7})
\begin{align*}
\underline m_2^0(z)=\frac{v_n+\underline m_2^0(z)c_n\int\frac{t^2}{|1+t\underline m_n^0(z)|^2}dH_n(t)}{\left|-z+c_n\int\frac{t}{1+t\underline m_n^0(z)}dH_n(t)\right|^2}
\end{align*}
and
\begin{align*}
\underline m_{n2}(z)=\frac{v_n+\underline m_{n2}(z)c_n\int\frac{t^2}{|1+t\underline m_n(z)|^2}dH_n(t)+\Im R_n}{\left|-z+c_n\int\frac{t}{1+t\underline m_n(z)}dH_n(t)-R_n\right|^2}.
\end{align*}
When $|\Im R_n|<v_n$, by the Cauchy-Schwarz inequality and the fact (see (3.17 in \cite{Bai1998No}))
\begin{align*}
\left(\frac{\underline m_2^0(z)c_n\int\frac{t^2}{|1+t\underline m_n^0(z)|^2}dH_n(t)}{v_n+\underline m_2^0(z)c_n\int\frac{t^2}{|1+t\underline m_n^0(z)|^2}dH_n(t)}\right)^{1/2}<1-Cv_n^2,
\end{align*}
we have
\begin{align*}
&\left|\frac{{c_n}\int\frac {t^2}{\left(\underline m_n(z)t+1\right)\left(\underline m_n^0(z)t+1\right)}dH_n(t)}{\left({c_n}\int\frac t{\underline m_n(z)t+1}dH_n(t)-z-R_n\right)\left({c_n}\int\frac t{\underline m_n^0(z)t+1}dH_n(t)-z\right)}\right|\\
\le&\left(\frac{{c_n}\int\frac {t^2}{\left|\underline m_n(z)t+1\right|^2}dH_n(t)}{\left|{c_n}\int\frac t{\underline m_n(z)t+1}dH_n(t)-z-R_n\right|^2}\right)^{1/2}
\left(\frac{{c_n}\int\frac {t^2}{\left|\underline m_n^0(z)t+1\right|^2}dH_n(t)}{\left|{c_n}\int\frac t{\underline m_n^0(z)t+1}dH_n(t)-z\right|^2}\right)^{1/2}\\
\le&\left(\frac{{c_n}\underline m_{n2}(z)\int\frac {t^2}{\left|\underline m_n(z)t+1\right|^2}dH_n(t)}{v_n+{c_n}\underline m_{n2}(z)\int\frac {t^2}{\left|\underline m_n(z)t+1\right|^2}dH_n(t)+\Im R_n}\right)^{1/2}
\left(\frac{{c_n}\underline m_2^0(z)\int\frac {t^2}{\left|\underline m_n^0(z)t+1\right|^2}dH_n(t)}{v_n+{c_n}\underline m_2^0(z)\int\frac {t^2}{\left|\underline m_n^0(z)t+1\right|^2}dH_n(t)}\right)^{1/2}\\
\le&
\left(\frac{{c_n}\underline m_2^0(z)\int\frac {t^2}{\left|\underline m_n^0(z)t+1\right|^2}dH_n(t)}{v_n+{c_n}\underline m_2^0(z)\int\frac {t^2}{\left|\underline m_n^0(z)t+1\right|^2}dH_n(t)}\right)^{1/2}\le1-Cv_n^2.
\end{align*}
Under the condition $|\Im R_n|<v_n$, one gets
\begin{align*}
\left|{\underline m_n(z)}-{\underline m_n^0(z)}\right|\le C^{-1}v_n^{-2}|{\underline m_n(z)}{\underline m_n^0(z)}R_n|=C^{-1}v_n^{-2}|c_nz{\underline m_n^0(z)}w_n|.
\end{align*}
\\
Letting $\mu_n=n^{6\eta/343}$, we can assert that when $|u|\le \mu_nv_n^{-1}$, $|w_n|\le \mu_n^{-1}v_n^5$, and $\lambda_{max}\le K\log n$, we have for large $n$, $|z|\le 2\mu_nv_n^{-1}$ and
\begin{align*}
|R_n|<v_n.
\end{align*}
In fact, on the set $\{\lambda_{max}\le K\log n\}$ and $|u|\le \mu_nv_n^{-1}$, we give a lower bound of $|\underline m_n(z)|$. When $-\mu_nv_n^{-1}\le u\le -v_n$ ot $\lambda_{\max}+v_n\le u\le \mu_nv_n^{-1}$,
\begin{align*}
|\underline m_n(z)|\ge|\Re\underline m_n(z)|\ge\frac{K\log n+\mu_nv_n^{-1}}{\left(K\log n+\mu_nv_n^{-1}\right)^2+v_n^2}\ge\frac1{2\mu_nv_n^{-1}}
\end{align*}
for large $n$. When $-v_n<u<\lambda_{\max}+v_n$,
\begin{align*}
|\underline m_n(z)|\ge|\Im\underline m_n(z)|\ge\frac{v_n}{\left(K\log n+v_n\right)^2+v_n^2}\ge\frac{v_n}{\mu_n}
\end{align*}
for all large $n$. These yield $|\underline m_n(z)|\ge\frac12\mu_n^{-1}v_n$. Therefore, when $|u|\le \mu_nv_n^{-1}$, $|w_n|\le \mu_n^{-1}v_n^5$, and $\lambda_{max}\le K\log n$, we have for large $n$, $|z|\le 2\mu_nv_n^{-1}$ and
\begin{align*}
|R_n|=|c_nzw_n/\underline m_n(z)|\le 4c_n\mu_nv_n^{-1}\mu_n^{-1}v_n^5\mu_nv_n^{-1}\le 4c_nv_n^{2}<v_n.
\end{align*}
Under the condition $|\Im R_n|<v_n$, one gets
\begin{align*}
\left|{\underline m_n(z)}-{\underline m_n^0(z)}\right|\le C^{-1}v_n^{-2}|c_nz{\underline m_n^0(z)}w_n|\le C\mu_n^{-1}v_n^2
\end{align*}
where the last inequality if from $|z\underline m_n^0(z)|\le1+C/v_n$.
\\
Thus, it follows that
\begin{align*}
\left|{\underline m_n(z)}-{\underline m_n^0(z)}\right|=&\left|{\underline m_n(z)}-{\underline m_n^0(z)}\right|I\left(|u|\le \mu_nv_n^{-1},|w_n|\le \mu_n^{-1}v_n^5,\lambda_{max}\le K\log n\right)\\
&+\left|{\underline m_n(z)}-{\underline m_n^0(z)}\right|I\left(|u|> \mu_nv_n^{-1},|w_n|\le \mu_n^{-1}v_n^5,\lambda_{max}\le K\log n\right)\\
&+\left|{\underline m_n(z)}-{\underline m_n^0(z)}\right|I\left(|w_n|> \mu_n^{-1}v_n^5 \ {\rm or} \ \lambda_{max}\le K\log n\right)\\
\le &C\mu_n^{-1}v_n^2+\frac{2}{\mu_nv_n^{-1}-K\log n}+\frac2{v_n}I\left(|w_n|> \mu_n^{-1}v_n^5 \ {\rm or} \ \lambda_{max}> K\log n\right)\\
\le &C\mu_n^{-1}v_n^1+\frac2{v_n}I\left(|w_n|> \mu_n^{-1}v_n^5 \ {\rm or} \ \lambda_{max}>K\log n\right).
\end{align*}
\\
In \ref{se1}, we obtain that for any subsets $S_n\subset\mathbb{R}$ containing at most $n$ elements, any $l\ge1/2$ and all $\varepsilon>0$,
\begin{align*}
&{\rm P}\left(\max_{u\in S_n}|w_n|v_n^{-5}>\varepsilon\right)\le C_l\varepsilon^{-2l}n^{2-2\eta l/49}
\end{align*}
which implies that
\begin{align*}
{\rm P}\left(\max_{u\in S_n}|w_n|\mu_nv_n^{-5}>\varepsilon\right)\le C_l\varepsilon^{-2l}n^{2-2\eta l/343}.
\end{align*}
From (\ref{al11}) and the above inequality, we conclude for large $n$ and any positive $\varepsilon$ and $l>0$,
\begin{align*}
&{\rm P}\left(v_n^{-1}\max_{u\in S_n}\left|{\underline m_n(z)}-{\underline m_n^0(z)}\right|>\varepsilon\right)\\
\le &C_l\varepsilon^{-l}\left[\mu_n^{-l}+{v_n}^{-2l}\left(\sum_{u\in S_n}{\rm P}\left(\mu_n v_n^{-5}|w_n|>1\right)+{\rm P}\left( \lambda_{max}>K\log n\right)\right)\right]\\
\le &C_l\varepsilon^{-l}\left[\mu_n^{-l}+{v_n}^{-2l}\left(n^{3-2\eta t/343}+n^{-t}\right)\right]\\
\le &C_l\varepsilon^{-l}n^{-6\eta l/343}
\end{align*}
where $t\ge45l+1029/2\eta$. It can be verified that for large $n$ and any positive $\varepsilon$ and $l$,
\begin{align}\label{al12}
&{\rm P}\left(v_n^{-1}\max_{u\in \mathbb{R}}\left|{\underline m_n(z)}-{\underline m_n^0(z)}\right|>\varepsilon\right)
\le C_l\varepsilon^{-l}n^{-6\eta l/343}.
\end{align}
\\
 Since for any $l>0$,
\begin{align*}
\re_k\left(v_n^{-l}\max_{u\in \mathbb{R}}\left|{\underline m_n(z)}-{\underline m_n^0(z)}\right|^l\right)
\end{align*}
for $k=0,\cdots,n$ forms a martingale, it follows that for any $t\ge1$,
\begin{align*}
\left[\re_k\left(v_n^{-l}\max_{u\in \mathbb{R}}\left|{\underline m_n(z)}-{\underline m_n^0(z)}\right|^l\right)\right]^t
\end{align*}
for $k=0,\cdots,n$ forms a submartingale. Therefore, for any $\varepsilon>0$, $t\ge1$, and $l>0$, from Lemma \ref{le7}, Lemma \ref{le8}, and (\ref{al12}) with $l$ replaced by $2tl$, we have
\begin{align*}
&{\rm P}\left(\max_{k\le n}\re_k\left(v_n^{-l}\max_{u\in \mathbb{R}}\left|{\underline m_n(z)}-{\underline m_n^0(z)}\right|^l\right)>\varepsilon\right)\\
\le&\varepsilon^{-t}\re\left(v_n^{-lt}\max_{u\in \mathbb{R}}\left|{\underline m_n(z)}-{\underline m_n^0(z)}\right|^{lt}\right) \\
\le&C_{lt}\varepsilon^{-t}n^{-6\eta lt/343}.
\end{align*}
From this and by taking $t>343/(6\eta l)$, it follows that for $l>0$
\begin{align}\label{al13}
\max_{k\le n}\re_k\left(v_n^{-l}\max_{u\in \mathbb{R}}\left|{\underline m_n(z)}-{\underline m_n^0(z)}\right|^l\right)\xrightarrow{a.s.}1.
\end{align}
\\
It can be verified from (\ref{al13}) for $l>0$
\begin{align*}
\max_{k\le n}\re_k\left(F^{\underline\bs_n}{[a,b]}\right)^l=o_{a.s.}(v_n^l)=o_{a.s.}(n^{-6\eta l/49})
\end{align*}
and
\begin{align}\label{al14}
\max_{k\le n}\re_k\left(F^{\underline\bs_n}{[a',b']}\right)^l=o_{a.s.}(v_n^l)=o_{a.s.}(n^{-6\eta l/49})
\end{align}
where the details can be seen in \cite{bai2010spectral}.
}
\end{supplement}

\newpage
\begin{supplement}[id=se1]
\sname{Supplement B}
\stitle{a convergence rate of $w_n$}
\sdescription{
For $z=u+iv_n=u+in^{-6\eta/49}$, we shall show the almost sure convergence of
\begin{align*}
\max_{u\in S_n}\frac{|w_n|}{v_n^5}
\end{align*}
to 0. Let
\begin{align*}
\underline m_{nj}(z)=-(1-c_n)/z+c_nm_{F^{\bs_{nj}}}(z),
\end{align*}
then one finds
\begin{align}\label{al8}
\max_{j\le n}|\underline m_n(z)-\underline m_{nj}(z)|=&\frac1n\max_{j\le n}|\rtr(\bd_n^{-1}-\bd_{nj}^{-1})|=\frac1n\max_{j\le n}|\rho_j\br_j^*\bd_{nj}^{-2}\br_j|\\
\le&\frac1n\max_{j\le n}\frac{\br_j^*\bd_{nj}^{-1}(\bd_{nj}^*)^{-1}\br_j}{\Im(\br_j^*\bd_{nj}^{-1}\br_j)}=\frac1{nv_n}.\notag
\end{align}
Rewrite
\begin{align*}
w_n=&-\frac1{p z}\sum_{j=1}^n\rho_j\Bigg\{\br_j^*\bd_{nj}^{-1}\left(\underline m_n(z)\bS_p+\bi_p\right)^{-1}\br_j
-\frac1n\rtr\left[\left(\underline m_n(z)\bS_p+\bi_p\right)^{-1}\bS_p\bd_n^{-1}\right]\Bigg\}\\
=&-\frac1{p z}\sum_{j=1}^n\rho_j\Bigg\{\br_j^*\bd_{nj}^{-1}\left(\underline m_n(z)\bS_p+\bi_p\right)^{-1}\br_j
-\br_j^*\bd_{nj}^{-1}\left(\underline m_{nj}(z)\bS_p+\bi_p\right)^{-1}\br_j\Bigg\}\\
&-\frac1{p z}\sum_{j=1}^n\rho_j\Bigg\{\br_j^*\bd_{nj}^{-1}\left(\underline m_{nj}(z)\bS_p+\bi_p\right)^{-1}\br_j-\frac1n\rtr\left[\left(\underline m_{nj}(z)\bS_p+\bi_p\right)^{-1}\bS_p\bd_{nj}^{-1}\right]\Bigg\}\\
&-\frac1{p z}\sum_{j=1}^n\rho_j\Bigg\{\frac1n\rtr\left[\left(\underline m_{nj}(z)\bS_p+\bi_p\right)^{-1}\bS_p\bd_{nj}^{-1}\right]-\frac1n\rtr\left[\left(\underline m_{nj}(z)\bS_p+\bi_p\right)^{-1}\bS_p\bd_{n}^{-1}\right]\Bigg\}\\
&-\frac1{p z}\sum_{j=1}^n\rho_j\Bigg\{\frac1n\rtr\left[\left(\underline m_{nj}(z)\bS_p+\bi_p\right)^{-1}\bS_p\bd_{n}^{-1}\right]-\frac1n\rtr\left[\left(\underline m_{n}(z)\bS_p+\bi_p\right)^{-1}\bS_p\bd_{n}^{-1}\right]\Bigg\}\\
\triangleq&-\frac1{p z}\sum_{j=1}^n\rho_j(d_j^1+d_j^2+d_j^3+d_j^4).
\end{align*}
It is easy to verify
\begin{align*}
\Im \br_j^*(z^{-1}\bs_{nj}-\bi_p)^{-1}\br_j\ge0.
\end{align*}
Therefore, for each $j$,
\begin{align*}
|z^{-1}\rho_j|\le\frac1{v_n}.
\end{align*}
Hence, it is sufficient to show the a.s.convergence of
\begin{align}\label{al9}
\max_{j\le n,u\in S_n}\frac{|d_j^k|}{v_n^6}
\end{align}
to 0 for $k=1,2,3,4$.
\\
From Lemma \ref{le2} (b) and (\ref{al8}), we deduce
\begin{align*}
|d_j^1|\le16v_n^{-2}|\underline m_n(z)-\underline m_{nj}(z)| \ \|\br_j\|_2^2 \ \|\bd_{nj}^{-1}\| \ \|\bS_p\|^3\le\frac{C}{nv_n^4}\|{\br}_j\|_2^2.
\end{align*}
Using Lemma \ref{le3}, it follows that, for any $\varepsilon>0$, $l\ge1/2$, and all large $n$,
\begin{align*}
&{\rm P}\left(\max_{j\le n,u\in S_n}\frac{|d_j^1|}{v_n^6}>\varepsilon\right)\le{\rm P}\left(\max_{j\le n}{\|{\br}_j\|_2^2}>\frac{\varepsilon nv_n^{10}}{C}\right)\\
\le&\sum_{j=1}^n{\rm P}\left({\|{\br}_j\|_2^2}>\frac{\varepsilon nv_n^{10}}{C}\right)
\le\left(\frac{C}{\varepsilon nv_n^{10}}\right)^{2l}\sum_{j=1}^n{\rm E}\left({\|{\br}_j\|_2^2}\right)^{2l}\\
\le&\left(\frac{C}{\varepsilon nv_n^{10}}\right)^{2l}n
\frac{n^{(1-\eta)2l}}{n^{2l}}\le\frac{C_ln^{1-2\eta l}}{\left(\varepsilon nv_n^{10}\right)^{2l}}\le C_l\varepsilon^{-2l}n^{1-2\eta l}.
\end{align*}
The last bound is summable when $l>\frac1{\eta}$, so we have $(\ref{al9})\xrightarrow{a.s.}0$ when $k=1$.
\\
Likewise, by Lemma \ref{le2} (a), one finds for any $l\ge1/2$,
\begin{align*}
\re|v_n^{-6}d_j^2|^{2l}\le\frac{C_l}{v_n^{12l}n^{2l}}n^{(1-\eta)2l}\left\|\bd_{nj}^{-1}\left(\underline m_{nj}(z)\bS_p+\bi_p\right)^{-1}\right\|^{2l}
\le\frac{C_l}{v_n^{16l}n^{2\eta l}}.
\end{align*}
We have then, for any $\varepsilon>0$ and $l\ge1/2$,
\begin{align*}
&{\rm P}\left(\max_{j\le n,u\in S_n}\frac{|d_j^2|}{v_n^6}>\varepsilon\right)\le C_l\varepsilon^{-2l}\frac{n^2}{v_n^{16l}n^{2\eta l}}\le C_l\varepsilon^{-2l}n^{2-2\eta l/49}
\end{align*}
which implies that $(\ref{al9})\xrightarrow{a.s.}0$ for $k=2$ by taking $l>\frac{147}{2\eta}$. Using Lemma \ref{le2} (a) and (\ref{al4}), we find
\begin{align*}
|v_n^{-6}d_j^3|=&\frac{1}{v_n^{6}n}\left|\rtr\left[\left(\underline m_{nj}(z)\bS_p+\bi_p\right)^{-1}\bS_p\left(\bd_{nj}^{-1}-\bd_{n}^{-1}\right)\right]\right|\\
=&\frac{1}{v_n^{6}n}\left|\rho_j\br^*_j\bd_{nj}^{-1}\left(\underline m_{nj}(z)\bS_p+\bi_p\right)^{-1}\bS_p\bd_{nj}^{-1}\br_j\right|\\
\le&\frac{4C}{v_n^{7}n}\left|\rho_j\br^*_j\bd_{nj}^{-1}\bS_p\bd_{nj}^{-1}\br_j\right|\le\frac{4C}{v_n^{8}n},
\end{align*}
so that $(\ref{al9})\xrightarrow{a.s.}0$ for $k=3$.
By Lemma \ref{le2} (a) and (\ref{al8}), we get
\begin{align*}
|v_n^{-6}d_j^4|=&\frac{1}{v_n^{6}n}|\underline m_{nj}(z)-\underline m_{j}(z)|\left|\rtr\left[\left(\underline m_{nj}(z)\bS_p+\bi_p\right)^{-1}\bS_p\left(\underline m_{n}(z)\bS_p+\bi_p\right)^{-1}\bS_p\bd_{n}^{-1}\right]\right|\\
\le&\frac{C}{v_n^{6}n}\times\frac1{nv_n}\times\frac{p}{v_n^3}=\frac{C}{n v_n^{10}}
\end{align*}
so that $(\ref{al9})\xrightarrow{a.s.}0$ for $k=4$.
\\
Thus, we deduce, for any $l\ge1/2$ and all $\varepsilon>0$,
\begin{align*}
&{\rm P}\left(\max_{u\in S_n}|w_n|v_n^{-5}>\varepsilon\right)\le C_l\varepsilon^{-2l}n^{2-2\eta l/49}.
\end{align*}
Therefore, $\max_{u\in S_n}|w_n|v_n^{-5}\xrightarrow{a.s.}0$ by taking $l>\frac{147}{2\eta}$.
}
\end{supplement}


\begin{thebibliography}{25}

\bibitem[\protect\citeauthoryear{Anderson}{1983}]{Anderson1983An}
\begin{bbook}[author]
\bauthor{\bsnm{Anderson},~\bfnm{T.~W.}\binits{T.~W.}}
(\byear{1983}).
\btitle{An Introduction to Multivariate Statistical Analysis, Second Edition,
  Wiley, New York}.
\end{bbook}
\endbibitem

\bibitem[\protect\citeauthoryear{Bai and Silverstein}{1998}]{Bai1998No}
\begin{barticle}[author]
\bauthor{\bsnm{Bai},~\bfnm{Z.~D.}\binits{Z.~D.}} \AND
  \bauthor{\bsnm{Silverstein},~\bfnm{Jack~W.}\binits{J.~W.}}
(\byear{1998}).
\btitle{No Eigenvalues Outside the Support of the Limiting Spectral
  Distribution of Large-Dimensional Sample Covariance Matrices}.
\bjournal{Annals of Probability}
\bvolume{26}
\bpages{316-345}.
\end{barticle}
\endbibitem

\bibitem[\protect\citeauthoryear{Bai and Silverstein}{1999}]{Bai1999Exact}
\begin{barticle}[author]
\bauthor{\bsnm{Bai},~\bfnm{Z.~D.}\binits{Z.~D.}} \AND
  \bauthor{\bsnm{Silverstein},~\bfnm{Jack~W.}\binits{J.~W.}}
(\byear{1999}).
\btitle{Exact Separation of Eigenvalues of Large Dimensional Sample Covariance
  Matrices}.
\bjournal{Annals of Probability}
\bvolume{27}
\bpages{1536-1555}.
\end{barticle}
\endbibitem

\bibitem[\protect\citeauthoryear{Bai and Silverstein}{2010}]{bai2010spectral}
\begin{bbook}[author]
\bauthor{\bsnm{Bai},~\bfnm{Zhidong}\binits{Z.}} \AND
  \bauthor{\bsnm{Silverstein},~\bfnm{Jack~William}\binits{J.~W.}}
(\byear{2010}).
\btitle{Spectral analysis of large dimensional random matrices}.
\bpublisher{Springer}.
\end{bbook}
\endbibitem

\bibitem[\protect\citeauthoryear{Bai and Yin}{1993}]{BaiYin1993}
\begin{barticle}[author]
\bauthor{\bsnm{Bai},~\bfnm{Z.~D.}\binits{Z.~D.}} \AND
  \bauthor{\bsnm{Yin},~\bfnm{Y.~Q.}\binits{Y.~Q.}}
(\byear{1993}).
\btitle{Limit of the Smallest Eigenvalue of a Large Dimensional Sample
  Covariance Matrix}.
\bjournal{The Annals of Probability}
\bvolume{21}
\bpages{pp. 1275-1294}.
\end{barticle}
\endbibitem

\bibitem[\protect\citeauthoryear{Bai and Zhou}{2008}]{Bai2008Large}
\begin{barticle}[author]
\bauthor{\bsnm{Bai},~\bfnm{Zhidong}\binits{Z.}} \AND
  \bauthor{\bsnm{Zhou},~\bfnm{Wang}\binits{W.}}
(\byear{2008}).
\btitle{Large sample covariance matrices without independence structures in
  columns}.
\bjournal{Statistica Sinica}
\bvolume{18}
\bpages{425-442}.
\end{barticle}
\endbibitem

\bibitem[\protect\citeauthoryear{Bai et~al.}{1988}]{Bai1988A}
\begin{barticle}[author]
\bauthor{\bsnm{Bai},~\bfnm{Z.~D.}\binits{Z.~D.}},
  \bauthor{\bsnm{Silverstein},~\bfnm{Jack~W.}\binits{J.~W.}},
  \bauthor{\bsnm{Yin},~\bfnm{Y.~Q.}\binits{Y.~Q.}} \AND
  \bauthor{\bsnm{De~Leeuw},~\bfnm{J}\binits{J.}}
(\byear{1988}).
\btitle{A note on the largest eigenvalue of a large dimensional sample
  covariance matrix}.
\bjournal{Journal of Multivariate Analysis}
\bvolume{26}
\bpages{166-168}.
\end{barticle}
\endbibitem

\bibitem[\protect\citeauthoryear{Burkholder}{1973}]{burkholder1973distribution}
\begin{barticle}[author]
\bauthor{\bsnm{Burkholder},~\bfnm{Donald~L}\binits{D.~L.}}
(\byear{1973}).
\btitle{Distribution function inequalities for martingales}.
\bjournal{the Annals of Probability}
\bvolume{1}
\bpages{19--42}.
\end{barticle}
\endbibitem

\bibitem[\protect\citeauthoryear{Chafa\"{\i} and
  Tikhomirov}{2017}]{Chafa2017On}
\begin{barticle}[author]
\bauthor{\bsnm{Chafa\"{\i}},~\bfnm{Djalil}\binits{D.}} \AND
  \bauthor{\bsnm{Tikhomirov},~\bfnm{Konstantin}\binits{K.}}
(\byear{2017}).
\btitle{On the convergence of the extremal eigenvalues of empirical covariance
  matrices with dependence}.
\bjournal{Probability Theory \& Related Fields}
\bpages{1-43}.
\end{barticle}
\endbibitem

\bibitem[\protect\citeauthoryear{Feldheim and Sodin}{2010}]{Feldheim2010A}
\begin{barticle}[author]
\bauthor{\bsnm{Feldheim},~\bfnm{Ohad~N.}\binits{O.~N.}} \AND
  \bauthor{\bsnm{Sodin},~\bfnm{Sasha}\binits{S.}}
(\byear{2010}).
\btitle{A Universality Result for the Smallest Eigenvalues of Certain Sample
  Covariance Matrices}.
\bjournal{Geometric \& Functional Analysis}
\bvolume{20}
\bpages{88-123}.
\end{barticle}
\endbibitem

\bibitem[\protect\citeauthoryear{Hall and Heyde}{1980}]{hall1980martingale}
\begin{bbook}[author]
\bauthor{\bsnm{Hall},~\bfnm{Peter}\binits{P.}} \AND
  \bauthor{\bsnm{Heyde},~\bfnm{Christopher~C.}\binits{C.~C.}}
(\byear{1980}).
\btitle{Martingale limit theory and its application}.
\bpublisher{Academic press New York}.
\end{bbook}
\endbibitem

\bibitem[\protect\citeauthoryear{Horn and Johnson}{1985}]{h1985}
\begin{bbook}[author]
\bauthor{\bsnm{Horn},~\bfnm{Roger~A.}\binits{R.~A.}} \AND
  \bauthor{\bsnm{Johnson},~\bfnm{Charles~R.}\binits{C.~R.}}
(\byear{1985}).
\btitle{Matrix Analysis}.
\bpublisher{Cambridge University Press}.
\end{bbook}
\endbibitem

\bibitem[\protect\citeauthoryear{Horn and Johnson}{1991}]{h1991}
\begin{bbook}[author]
\bauthor{\bsnm{Horn},~\bfnm{Roger~A.}\binits{R.~A.}} \AND
  \bauthor{\bsnm{Johnson},~\bfnm{Charles~R.}\binits{C.~R.}}
(\byear{1991}).
\btitle{Topics in Matrix Analysis}.
\bpublisher{Cambridge University Press}.
\end{bbook}
\endbibitem

\bibitem[\protect\citeauthoryear{Jonsson}{2008}]{Jonsson2008Some}
\begin{barticle}[author]
\bauthor{\bsnm{Jonsson},~\bfnm{Dag}\binits{D.}}
(\byear{2008}).
\btitle{Some limit theorems for the eigenvalues of a sample covariance matrix}.
\bjournal{Journal of Multivariate Analysis}
\bvolume{12}
\bpages{1-38}.
\end{barticle}
\endbibitem

\bibitem[\protect\citeauthoryear{Li and Yao}{2017}]{Li2017On}
\begin{barticle}[author]
\bauthor{\bsnm{Li},~\bfnm{Weiming}\binits{W.}} \AND
  \bauthor{\bsnm{Yao},~\bfnm{Jianfeng}\binits{J.}}
(\byear{2017}).
\btitle{On structure testing for component covariance matrices of a high
  dimensional mixture}.
\bjournal{Journal of the Royal Statistical Society}.
\end{barticle}
\endbibitem

\bibitem[\protect\citeauthoryear{Marchenko and
  Pastur}{1967}]{marchenko1967distribution}
\begin{barticle}[author]
\bauthor{\bsnm{Marchenko},~\bfnm{Vladimir~Alexandrovich}\binits{V.~A.}} \AND
  \bauthor{\bsnm{Pastur},~\bfnm{Leonid~Andreevich}\binits{L.~A.}}
(\byear{1967}).
\btitle{Distribution of eigenvalues for some sets of random matrices}.
\bjournal{Matematicheskii Sbornik}
\bvolume{114}
\bpages{507--536}.
\end{barticle}
\endbibitem

\bibitem[\protect\citeauthoryear{Paul and Silverstein}{2009}]{Paul2009No}
\begin{barticle}[author]
\bauthor{\bsnm{Paul},~\bfnm{Debashis}\binits{D.}} \AND
  \bauthor{\bsnm{Silverstein},~\bfnm{Jack~W.}\binits{J.~W.}}
(\byear{2009}).
\btitle{No eigenvalues outside the support of the limiting empirical spectral
  distribution of a separable covariance matrix}.
\bjournal{Journal of Multivariate Analysis}
\bvolume{100}
\bpages{37-57}.
\end{barticle}
\endbibitem

\bibitem[\protect\citeauthoryear{Pillai and Yin}{2014}]{Pillai2014Universality}
\begin{barticle}[author]
\bauthor{\bsnm{Pillai},~\bfnm{Natesh~S.}\binits{N.~S.}} \AND
  \bauthor{\bsnm{Yin},~\bfnm{Jun}\binits{J.}}
(\byear{2014}).
\btitle{Universality of covariance matrices.}
\bjournal{Annals of Applied Probability}
\bvolume{24}
\bpages{935-1001}.
\end{barticle}
\endbibitem

\bibitem[\protect\citeauthoryear{Vershynin}{2010}]{ver2010}
\begin{barticle}[author]
\bauthor{\bsnm{R},~\bfnm{Vershynin}\binits{V.}}
\btitle{Introduction to the non-asymptotic analysis of random matrices}.
\bjournal{arXiv preprint}
\bvolume{arXiv:1011.3027}.
\end{barticle}
\endbibitem

\bibitem[\protect\citeauthoryear{Sandrine}{2009}]{P2009Universality}
\begin{barticle}[author]
\bauthor{\bsnm{Sandrine},~\bfnm{P\'{e}ch\'{e}}\binits{P.}}
(\byear{2009}).
\btitle{Universality results for the largest eigenvalues of some sample
  covariance matrix ensembles}.
\bjournal{Probability Theory \& Related Fields}
\bvolume{143}
\bpages{481-516}.
\end{barticle}
\endbibitem

\bibitem[\protect\citeauthoryear{Silverstein}{1995}]{Silverstein1995Strong}
\begin{bbook}[author]
\bauthor{\bsnm{Silverstein},~\bfnm{Jack~W.}\binits{J.~W.}}
(\byear{1995}).
\btitle{Strong convergence of the empirical distribution of eigenvalues of
  large dimensional random matrices}.
\bpublisher{Academic Press, Inc.}
\end{bbook}
\endbibitem

\bibitem[\protect\citeauthoryear{Silverstein and
  Bai}{1995}]{silverstein1995empirical}
\begin{barticle}[author]
\bauthor{\bsnm{Silverstein},~\bfnm{Jack~W}\binits{J.~W.}} \AND
  \bauthor{\bsnm{Bai},~\bfnm{ZD}\binits{Z.}}
(\byear{1995}).
\btitle{On the empirical distribution of eigenvalues of a class of large
  dimensional random matrices}.
\bjournal{Journal of Multivariate analysis}
\bvolume{54}
\bpages{175--192}.
\end{barticle}
\endbibitem

\bibitem[\protect\citeauthoryear{Tikhomirov}{2015}]{Tikhomirov2015The}
\begin{barticle}[author]
\bauthor{\bsnm{Tikhomirov},~\bfnm{Konstantin}\binits{K.}}
(\byear{2015}).
\btitle{The limit of the smallest singular value of random matrices with i.i.d.
  entries}.
\bjournal{Advances in Mathematics}
\bvolume{284}
\bpages{1-20}.
\end{barticle}
\endbibitem

\bibitem[\protect\citeauthoryear{Wachter}{1978}]{Wachter1978The}
\begin{barticle}[author]
\bauthor{\bsnm{Wachter},~\bfnm{Kenneth~W.}\binits{K.~W.}}
(\byear{1978}).
\btitle{The Strong Limits of Random Matrix Spectra for Sample Matrices of
  Independent Elements}.
\bjournal{Annals of Probability}
\bvolume{6}
\bpages{1-18}.
\end{barticle}
\endbibitem

\bibitem[\protect\citeauthoryear{Yin, Bai and Krishnaiah}{1988}]{Yin1988}
\begin{barticle}[author]
\bauthor{\bsnm{Yin},~\bfnm{Y.~Q.}\binits{Y.~Q.}},
  \bauthor{\bsnm{Bai},~\bfnm{Z.~D.}\binits{Z.~D.}} \AND
  \bauthor{\bsnm{Krishnaiah},~\bfnm{P.~R.}\binits{P.~R.}}
(\byear{1988}).
\btitle{On the limit of the largest eigenvalue of the large dimensional sample
  covariance matrix}.
\bjournal{Probability Theory and Related Fields}
\bvolume{78}
\bpages{pp. 509-521}.
\bdoi{10.1007/BF00353874}
\end{barticle}
\endbibitem

\end{thebibliography}

\end{document}